\documentclass{amsart}

\usepackage{amssymb}
\usepackage{amsmath}
\usepackage{enumitem}
\usepackage{hyperref}

\usepackage{tikz-cd}
\usepackage{pgfplots}

\title{On the timescale at which statistical stability breaks down}
\author{Neil Dobbs}
\address{University College Dublin, Ireland.}
\thanks{N.D.\ was supported by the ERC Bridges grant while at the University of Geneva. This work sprang from the authors' encounter during the 2016 ESI programme ``Mixing Flows and Averaging Methods" in Vienna.}
\email{neil.dobbs@ucd.ie}
\author{Alexey Korepanov}
\address{Exeter University, UK.}
\email{a.korepanov@exeter.ac.uk}
\thanks{A.K.\ was funded by a European Advanced Grant {\em StochExtHomog} (ERC AdG 320977) at the University of Warwick
and an Engineering and Physical Sciences Research Council grant EP/P034489/1 at the University of Exeter. 
A.K.\ is grateful to Mark Holland and Ian Melbourne for support and to University College Dublin for
hospitality during his visit.}

\thanks{The authors are very grateful to the referees for their thorough reading and helpful comments.}

\newcommand\ol{\overline}

\newcommand\pot{{\varphi}}   % observable or potential version of phi
\newcommand\sbar{{\overline{S}}}

\newcommand\eps{\varepsilon}
\DeclareMathOperator{\dist}{dist}
\DeclareMathOperator{\diam}{diam}

\DeclareMathOperator{\ind}{index}
\DeclareMathOperator{\hei}{height}
\DeclareMathOperator{\E}{\mathbb{E}}

\newcommand{\Lip}{{\mathrm{Lip}}}

\newcommand\arr{\mathbb{R}}

\newcommand\R{\mathbb{R}}
\newcommand\Z{\mathbb{Z}}

\newcommand\N{\mathbb{N}}

\newtheorem{thm}{Theorem}[section]
\newtheorem{dfn}[thm]{Definition}
\newtheorem{lem}[thm]{Lemma}
\newtheorem{prop}[thm]{Proposition}
\newtheorem{cor}[thm]{Corollary}

\newtheorem{rem}[thm]{Remark}
\newtheorem{example}[thm]{Example}

\numberwithin{equation}{section}

\usetikzlibrary{decorations.markings}
\tikzset{negated/.style={
        decoration={markings,
            mark= at position 0.5 with {
                \node[transform shape] (tempnode) {$\backslash$};
            }
        },
        postaction={decorate}
    }
}

\newcommand\hf{{\hat{f}}}
\newcommand\hF{{\hat{F}}}
\newcommand\hmu{{\hat{\mu}}}
\newcommand\hrho{{\hat{\rho}}}

\newcommand\cA{\mathcal{A}}

\newcommand\cC{{\mathcal C}}

\newcommand\cP{\mathcal{P}}

\def\mruf{{Misiurewicz-rooted unimodal family}}
\def\mc{\simeq} 
\def\ggg{\gtrsim} 
\def\lll{\lesssim}

\newcommand\rthm[1]{{Theorem~\ref{thm#1}}}

\newcommand\rprop[1]{{Proposition~\ref{prop#1}}}
\newcommand\rlem[1]{{Lemma~\ref{lem#1}}}

\newcommand\rdef[1]{{Definition~\ref{def#1}}}

\newcommand\rsec[1]{{\S\ref{sect#1}}}

\newcommand\rrem[1]{{Remark~\ref{rem#1}}}

\begin{document}
\begin{abstract}
        In dynamical systems, understanding statistical properties shared by most orbits and how these properties depend on the system are basic and important questions. Statistical properties may persist as one perturbs the system  (\emph{statistical stability} is said to hold), or may vary wildly. 
    The latter case is our subject of interest, and we ask at what timescale does statistical stability break down. 
 This is the time needed to observe, with a certain probability, a substantial difference in the statistical properties as described by (large but finite time) Birkhoff averages.
    %, or how long one needs to observe the evolution of a perturbed system to notice a substantial difference in statistical properties. As a descriptor of statistical behaviour, we consider Birkhoff averages on a long but finite timescape.

    The quadratic (or logistic) family is a natural and fundamental example where statistical stability does not hold. We study this family. When the base parameter is of Misiurewicz type, we show, sharply,  that if the parameter changes by $t$, it is necessary and sufficient to observe the system for a time at least of the order of $|t|^{-1}$ to see the lack of statistical stability.
\end{abstract}
    
\maketitle

\section{Introduction}
In this paper, we investigate the timescale at which statistical stability of dynamical systems breaks down. 
We carry out this study in the quadratic family, a standard test-bed for new directions in dynamics. 
The main theorems are stated in \rsec{Statements}.

A real-world system can be represented by a phase space $X$, the set of all possible configurations of the system. Its evolution, with discrete time-steps, is described by a map $f \colon X \to X$. 
Suppose $X$ is a Riemannian manifold and $f$ is continuous.  If $x,y \in X$ are nearby points, their orbits $x, f(x), f^2(x),\ldots$ and $y, f(y), \ldots$ remain close for a time. If the map is expanding, these orbits diverge in a time of  the order of $\log \dist(x,y)^{-1}$ and may have very different properties. It is then natural to look at statistical properties of orbits, for example by studying Birkhoff averages
$$
\sbar_n\pot (x) = \frac1n \sum_{j=0}^{n-1} \pot \circ f^j(x),$$
where $\pot \colon X \to \arr$ is a continuous function (called an \emph{observable}).

Perhaps surprisingly, in well-behaved systems, for a given $\pot$, the Birkhoff averages may converge as $n\to \infty$ to the same limit for almost every $x$ with respect to the volume measure on $X$. Better still, there is a unique $f$-invariant probability measure $\mu$ with the property that the limit is $\int \pot \, d\mu$ for \emph{every} continuous $\pot$.

\subsection{Structural stability.}
Suppose we have a smooth one-parameter family of (discrete-time) maps $f_t : X \to X$ for $t$ in a neighbourhood of $0$. The dynamics of nearby maps is relevant to the resilience to perturbation or if there is some uncertainty as to the governing parameters. 
If $$\dist(f_t(x), f_0(x)) \approx t,$$ as is reasonable, the orbits of $x$ under $f_0$ and $f_t$ are expected to diverge in approximately $\log |t|^{-1}$ time steps. Thus, comparing orbits of the same point under nearby maps does not lead very far. To deal with this, Andronov and Pontryagin \cite{AndPon} introduced the notion of \emph{structural stability,} when for each nearby map there exists a global homeomorphism which maps  orbits of the nearby map to orbits of the original. This concept works well for flows on compact surfaces \cite{Peix, SmaleWiGA} and more general Morse-Smale systems, for example.

Structural stability is a rather rigid property. A fundamental example where it fails is the family of quadratic (or \emph{logistic}) maps 
$$f_t : x \mapsto x^2 +(a+t),$$ 
where $a+t$ lies in the parameter interval  $[-2,1/4].$ 
From Jakobson's Theorem 
\cite{Jak}, 
one deduces that the topological entropy of $f_t$ is not locally constant at $t=0$ for any $a$ in a positive-measure set of parameters. In particular, structural stability does not hold.

\subsection{Statistical stability.}
Even without structural stability, statistical properties may appear to persist.
Suppose that $X$ is compact and let $m$ denote the volume measure on $X$,
normalized so that $m(X) = 1$.
An $f$-invariant probability measure $\mu$ on $X$ is called \emph{physical},
or \emph{Sinai-Ruelle-Bowen (SRB)}, if there exists $A \subset X$
with $m(A) > 0$ so that for all continuous $\pot \colon X \to \R$
and $x \in A$,
\[
  \lim_{n \to \infty} \bar{S}_{n} \pot (x)
  = \int \pot \, d\mu
  .
\]
If $m(A) = 1$, we say that $\mu$ is a \emph{global physical measure.}

We say that the family $f_t$ is \emph{statistically stable,} if
for every $f_t$ there exists a global physical measure $\mu_t$, and
for each continuous $\pot \colon X \to \arr$,
\[
  \lim_{t \to 0} \int \pot \,d\mu_t
  = \int \pot \, d\mu_0.
\]

Statistical stability has been studied by Keller~\cite{KelSSiscds}, Dolgopyat~\cite{DolOdomcd}, Alves and Viana~\cite{AlvesVianaSS}, Alves, Carvalho and Freitas~\cite{AlCaFr}, Freitas and Todd~\cite{FToddSS}
and others. The study of higher regularity properties was driven by Ruelle and Baladi, see \cite{RueDiff, RueStruct, BalBenSch} and references therein.

% By \cite{BlokhLyu}, if a quadratic map has a physical measure, then it is unique.
In the quadratic family, statistical stability holds at hyperbolic parameters (those corresponding to maps with periodic attractors). However, it does not hold everywhere, failing at most non-hyperbolic parameters \cite{Thun, DobTod}, even near the so-called Misiurewicz parameters \cite{DobTod}.
Moreover, there are quadratic maps \cite{HofKel1990a} for which there is no physical measure
to begin with.

  \begin{rem}
  One can obtain highly non-trivial positive results concerning statistical stability \cite{Tsu95, FToddSS}, and even H\"older continuity of the map $t \mapsto \int \pot d\mu_t$ \cite{BalBenSch}, if the parameter range is restricted to a nowhere dense, but positive measure, set.
  \end{rem}

  \subsection{The breakdown of statistical stability.}
Introducing $t$-dependence to our Birkhoff averages, we set
\[
  \bar{S}_{t,n} \pot 
  = \frac{1}{n} \sum_{j=0}^{n-1} \pot \circ f_t^j.
\]
For each $t,n$, we view $\bar{S}_{t,n} \pot$ as random a variable on the probability space $(X,m)$.
We suppose that $f_0$ admits a global physical measure $\mu_0$, and we use $\mu_t$ to refer to the global physical measures for $f_t$, whenever they exist.

  Consider the following diagram.
\[
\begin{tikzcd}[column sep=4cm, row sep=huge]
\sbar_{t,n}\pot (x) \arrow{r}{n\to\infty \quad m \text{ a.s.}} \arrow[swap]{d}{t\to 0} \arrow[swap, bend right=10, sloped]{dr}{  (t,n) \to (0,\infty)? } & \int \pot \, d\mu_t \arrow[negated]{d}{\, \quad t \to 0 }  \text{ if $\mu_t$ exists} \\
    \sbar_{0,n} \pot(x) \arrow{r}{n\to\infty \quad m \text{ a.s.}} & \int \pot \, d\mu_0  
  \end{tikzcd}
\]
  Following the lower-left path, 
$$
\lim_{n\to\infty} \lim_{t \to 0} \sbar_{t,n}\pot(x) 
  = \int \pot \, d\mu_0 \quad m\text{-almost surely}. $$
  Switch the order of limits and this will no longer hold. 
  The measures $\mu_t$ need not exist, and even restricting to parameters for which they do, the integrals $\int \pot \, d\mu_t$ need not vary continuously. 

  Now consider the diagonal arrow. Let $n(t)$ be an integer-valued function of $t$ with $n(t) \to \infty$ as $t \to 0$.
  %, and consider the limiting bevaviour of $\sbar_{t,n(t)} \pot$.
  Intuitively, if  $n(t) \ll -\log|t|$, then orbits of a point $x$ under $f_{0}$ and under $f_t$ do not have time to meaningfully diverge, so $\sbar_{t,n(t)}\pot \approx \sbar_{0,n(t)}\pot$ and
\begin{equation}\label{eqnAA0bis}
\lim_{t\to 0} \bar{S}_{t,n(t)} \pot
  =
  \int \pot \, d\mu_0
  \quad \text{$m$ almost-surely.} 
\end{equation}
  As a corollary,
\begin{equation}\label{eqnAA0}
      \bar{S}_{t,n(t)} \pot
  \to
  \int \pot \, d\mu_0
  \quad \text{in probability (w.r.t.\ $m$), as } \ t \to 0.
\end{equation}
The almost sure convergence \eqref{eqnAA0bis} is a rather rigid concept,
it is expected to break down once $n(t) \gg -\log|t|$, see \cite[Section~7]{KKM15}.

In this paper, we examine how fast $n(t)$ can grow without destroying the convergence in probability \eqref{eqnAA0}. Or, given the size of a small perturbation, we determine the minimum amount of observation time needed to discover instability in the statistical behaviour. 
Similarly, if we have some uncertainty in the parameter governing the system, the predicted statistical behaviour is valid up until some timescale. 

For the quadratic family, if the  base parameter is of Misiurewicz type, the statistical stability continues to hold as long as $n(t)$ grows more slowly than $t^{-1}$, see \rthm{Con}. This result is sharp: if $n(t)$ grows as fast as $t^{-1}$, continuity is lost, see \rthm{Discon}. 
We say that, in this context, \begin{quote}\emph{statistical stability breaks down at the timescale $\frac1t$.}\end{quote} 

    \subsection{Fast-slow systems.}
An initial stimulus for our work was the study of \emph{fast-slow} systems of the form:
\begin{equation}
  \begin{cases}
      s_{\eps,n+1} = s_{\eps,n} + \eps \pot(s_{\eps, n}, x_{\eps,n}), & s_{\eps,0} = 0\\
     x_{\eps,n+1} = f_\eps (x_{\eps,n}), & x_{\eps,0} \sim  m
  \end{cases}
\end{equation}
with $\eps \in [0,\eps_0]$.
When the maps $f_\eps$ are nonuniformly expanding,
under rather general assumptions it is proved~\cite{KKM15}
that as $\eps \to 0$, the random process
$s_{\eps, \lfloor \eps ^{-1} t \rfloor}$, $t \in [0,1]$,
converges in distribution to the solution of the ordinary
differential equation $\dot{s} = \int \pot(s, x) \, d\mu_0(x)$, $s(0) = 0$,
where $\mu_0$ is the physical measure for $f_0$.

In the case of logistic maps,
% with $a$ a Collet-Eckmann parameter
to satisfy the assumption that the maps $f_\eps$ are nonuniformly expanding,
the range of $\eps$ has to be restricted to a nowhere dense subset of $[0,\eps_0]$.
It is an interesting question whether the
restriction on parameters can be removed.
The authors of~\cite{KKM15} were asked this question
by various people, including D.~Dolgopyat and
the anonymous referee of~\cite{KKM15}. 

To simplify the model, we suppose that $\varphi$ does not depend on $s$,
i.e.\ $\varphi(s,x) = \varphi(x)$.
Then
\[
  s_{\eps, \lfloor \eps ^{-1} t \rfloor}
  = \eps \sum_{j=0}^{\lfloor \eps ^{-1} t \rfloor - 1}
  \pot \circ f_\eps^j.
\]
%is a finite sum, oblivious of possible
%complications in the long term dynamics,
%such as absence of the physical measures.
Our theorems 
respond to the above question, showing that convergence
breaks down without a restriction on the parameter range
but, surprisingly, for all shorter (and less natural) timescales,
one does have convergence.

\subsection{Stochastic stability.}

In this paper we perturb a dynamical system by considering
another one close to the original. Such perturbations are called
\emph{deterministic.} Another type is \emph{stochastic,}
where at each step a small perturbation is chosen randomly.

Suppose the base map has a physical measure $\mu_0$.
If the statistics of stochastically perturbed systems can be
described by measures $\mu_\eps$, where $\eps$ reflects the average strength of the
perturbation, and if $\mu_\eps \to \mu_0$ as $\eps \to 0$,
then the base map is \emph{stochastically stable.}
The question of stochastic stability has been treated successfully in
\cite{AlvVil13,BalBenDes02,BalVia96,BenYou92,KelSSiscds,LiWang13,Shen13,ShenStrien13}
among others.

In sharp contrast with statistical stability, almost every quadratic map 
is stochastically stable \cite{BalVia96,BenYou92, Shen13, Lyu2002Ann}.

%\subsection{Statistical detection of the lack of linear response.}
%While preparing this manuscript for submission, we became aware of an interesting article \cite{GWW16} by Gottwald \emph{et al} which examines the possibility of detecting numerically (or not) the lack of linear response in the quadratic family. 
%Their statistical methods are based on Birkhoff averages and hence are related to the timescales of our work. The find that, for a `global observable' $\pot(x) = x$, the timescale necessary for detection is $t^{-0.91}$ for a perturbation of size $t$, while choosing more focused observables allows detection at much shorter timescales. As linear response is stronger than statistical stability, these observations are compatible with our results. There is a more in-depth discussion in \cite{GWW16} about the implications of these results for modelling the real world. 

\subsection{Statistical detection of the lack of linear response.}
While preparing this manuscript for submission, we became aware of an interesting article \cite{GWW16} by Gottwald \emph{et al} which examines, via numerical experiments, the possibility of detecting statistically the lack of linear response in the quadratic family with a `global observable'. 

Linear response is a stronger property than statistical stability. Even so, it is found in \cite{GWW16} that
detection of absence of linear response requires 
a well-designed statistical test and observations on long timescales (such as $10^6$ iterations). 

The results of \cite{GWW16} suggest that one needs a timescale of order at least $t^{-0.91}$  to detect the lack of linear response under perturbations of size $t$ with a global observable. The timescale can be reduced by crafting special observables. These observations are compatible with our results concerning statistical stability. There is a more in-depth discussion in \cite{GWW16} about the implications for mathematical modelling.

\subsection{Organisation.}
The paper is organized as follows. In \rsec{Statements} we give formal definitions and statements of our main results. In \rsec{Prel} we assemble various results about the maps $f_t$ close to the base map $f_0$. In~\rsec{SetUp} we study topological and metric properties of first return maps to carefully chosen small neighbourhoods of the critical point. 

In~\rsec{Discon} we prove the lack of statistical stability on the timescale $n(t) = t^{-1}$. We find parameters $t_n$ with the critical point a super-attracting periodic point with period as short as possible. The size of the immediate basin of attraction of the critical point happens to be of the order of $t_n$. For any $C>0$, we show that a definite proportion of points fall into the basin in fewer than $Ct_n^{-1}$ iterates, which is enough to obliterate statistical stability. 

In~\rsec{Con} we prove statistical stability on shorter timescales. There is a natural argument which works for timescales up to $o(t^{-1/2})$ (see \rrem{SquareRootT}), but this is not optimal. To reach the optimal $o(t^{-1})$, we intricately construct an induced map.  We use it to approximate each $f_t$ with a non-uniformly expanding map for which martingale approximations give strong control of statistical properties.

\section{Statements} \label{sectStatements}
We shall often write $Df$ for the derivative $f'$ of a map $f$. 
 \begin{dfn}
\label{defUni}
We say that a continuous map $f:I \to I$, defined on a compact interval $I$, is \emph{unimodal} if $f$ has exactly one turning point $c$. We say $f$ is a \emph{smooth unimodal map} if, moreover, $f$ is continuously differentiable and $c$ is the unique (\emph{critical}) point satisfying $f'(c)=0$. The critical point and the map are \emph{non-degenerate} if $f''(c) \ne 0$. 
\end{dfn}
\begin{dfn}
    A map $f : I \to I$  is \emph{S-unimodal} if it is a $\cC^2$ smooth unimodal map with  critical point $c$,  $|f'|^{-1/2}$ is convex on each component of $I\setminus \{c\}$, $f(\partial I) \subset \partial I$ and $|f'| >1$ on $\partial I$. 
\end{dfn}

The convexity condition is equivalent (\cite{NowSan}, \cite[p.~266]{deMvS}), for $\cC^3$ maps, to having non-positive Schwarzian derivative, while strict convexity corresponds to negative Schwarzian derivative. Quadratic maps have negative Schwarzian derivative. A forward-invariant compact set $X$ for $f$ is \emph{hyperbolic repelling} if there exists $k\geq 1$ with $|Df^k|\geq2$ on $X$.
The \emph{post-critical orbit} is the set $\{f^n(f(c))\}_{n\geq0}$. 
 
\begin{dfn}
A smooth unimodal map is called \emph{Misiurewicz} if the closure of its post-critical orbit is a hyperbolic repelling set. 
\end{dfn}
Misiurewicz maps have strong expansion properties which outweigh any contraction caused by passage close to the critical point. By Singer's Theorem \cite[Theorem~III.1.6]{deMvS}, all periodic points of an $S$-unimodal Misiurewicz map are hyperbolic repelling. We shall recall further properties anon. 

Throughout the paper we fix $I = [-1,1]$, and all our unimodal maps have the critical point at $0$.

\begin{dfn}
    A \emph{\mruf} is a family $\{f_t\}_{t \in [0,\eps]}$, $\eps > 0$,
    of non-degenerate S-unimodal maps on $I$ with the critical point $0$.
    We require that $f_0$ is a Misiurewicz map and $f_t(x)$ is $\cC^2$ as a function
    of $(x,t)$.
\end{dfn}

\begin{dfn}
    We say that a \mruf{} $\{f_t\}$ is \emph{transversal} if
    $$
      \sum_{j=0}^\infty
      \frac{
        \partial_t f_t\bigl(f_0^j(0)\bigr) \bigr|_{t=0}
      }{
        (f^j_0)'\bigl(f_0(0)\bigr)
      } \ne 0.
    $$
\end{dfn}

Suppose that $\{f_t\}$ is a \mruf{} and let $\mu_0$ be the
unique $f_0$-invariant absolutely continuous probability measure \cite{Mis:IHES}.
Let $\pot \colon I \to \arr$ be a continuous observable and define
\[
    \sbar_{t,n}\pot := \frac{1}{n} \sum_{j=0}^{n-1} \pot \circ f_t^j .
\]
Let $\bar{\pot} = \int \pot \, d\mu_0$. Let $m$ denote the normalized to probability
Lebesgue measure on $I$.

\begin{thm}[Persistence of statistical stability] \label{thmCon}
    For any function $n \colon \arr^+ \to \Z^+$ such that
    $\lim_{t \to 0^+} n(t) = \infty$ and 
    $\lim_{t\to 0^+} t n(t) = 0,$
    $$
      \lim_{t \to 0^+} \int_{I} | \sbar_{t, n(t)}\pot - \bar{\pot} | \, dm = 0.$$
\end{thm}

\begin{thm}[Breakdown of statistical stability] \label{thmDiscon}
    Let $a>0$.
    If the family $\{f_t\}$ is transversal, then there exists a continuous observable $\pot$ for which 
    $$
    \limsup_{t \to 0^+} \int_{I} \sbar_{t, \lfloor \frac{a}{t} \rfloor} \pot \, dm
      \ne \bar{\pot}.
    $$
\end{thm}

\begin{rem}

    One could ask whether being $S$-unimodal is necessary or whether just assuming $\cC^2$ would suffice to prove these results. 
    We principally use the $S$-unimodal convexity condition to simplify control of distortion. 
    Ma\~n\'e's Hyperbolicity Theorem \cite{Mane:Hyp} for $\cC^2$ maps gives expansion and distortion control for the dynamics outside a neighbourhood of the critical point. So another path exists, but it would take more work and we wished to avoid further complicating an already technical paper. 
    Similarly, one could ask what happens for other, non-Misiurewicz, base parameters, for example Collet-Eckmann ones. 
    \end{rem}

\begin{example}
    Let $g_t(x) = x^2 + t_0 + t$ be a parametrisation of the quadratic family,
    with $t_0$ a Misiurewicz parameter in $[-2,1/4)$ and $t \in [0, 1/4 - t_0)$.
    Noting that $\partial_t g_t(x) \equiv 1$, transversality has been shown by
    Levin \cite{Lev}  (under more general
    summability conditions). This family does not leave $[-1,1]$ invariant,
    so it is not (quite) a \mruf. However, it can be transformed by a smooth family of affine transformations into a transversal \mruf. Hence our main theorems apply to the family $g_t$. 
    \iffalse
    for every $t$ there is the maximal interval
    $I_t = [-r_t, r_t]$ for which $g_t(I_t) \subset I_t$, and
    $$t \mapsto r_t = 1 + \sqrt{1-4(t_0+t)}$$
    is smooth near $0$.
    Rescaling by $r_t$, we obtain a conjugate quadratic family 
    $f_t(x) = r_t^{-1} g_t(r_tx)$ which is a \mruf.
    Writing $F(x,t) = f_t(x)$, $G(x,t) = g_t(x)$ and using $\partial_1, \partial_2$ to
    denote the partial derivatives with respect to the first and second coordinates,
    we compute 
    $$
      \partial_2 F(x,t) = r_t^{-1} \partial_2 G(r_tx,t) + r_t'r_t^{-1}(x\partial_1G(r_tx,t) - F(x,t)).
    $$
    Calculation gives
        $$
        \frac{\partial_2F( f_t^j(0), t)}{Df_t^j(f_t(0))} = 
         r_t^{-1} \frac{1}{Df^j_t(f_t(0))}
         +
        \frac{r_t'}{r_t} \left( \frac{f_t^j(0)}{Df_t^{j-1}(f_t(0))} - \frac{f_t^{j+1}(0)}{Df_t^j(f_t(0))} \right). 
        $$
    Summing with $t=0$, the telescopic sum contributes zero, while $Df^j_t(x) = Dg_t^j(r_tx)$; consequently transversality of $G$ implies transversality of $F$, as one would expect. Hence \rthm{Con} and \rthm{Discon} apply to $F$. One can then deduce corresponding statements for $G$.  
\fi
\end{example}

\section{Preliminaries} \label{sectPrel}

  We shall use the notation $A(\cdot) = O(B(\cdot))$ and
  $A(\cdot) \lll B(\cdot)$ interchangeably, meaning that there exists
  a constant $C > 0$ such that $A(\cdot) \leq C B(\cdot)$ for all
  sufficiently large (or small) values of the argument.
  If both $A(\cdot) \lll B(\cdot)$ and $B(\cdot) \lll A(\cdot)$, we write
  $A(\cdot) \mc B(\cdot)$.

\begin{dfn} \label{defExt}
  Let $W,V$ be open intervals. Suppose that $g \colon W \to V$ is a
  $\cC^2$ surjective diffeomorphism with $|Dg|^{-1/2}$ convex.
  Suppose that $g$ can be extended to a   $\cC^2$ surjective diffeomorphism
  $g \colon \hat W \to \hat V$ with $|Dg|^{-1/2}$ convex, where
  $\hat W, \hat V$ are intervals and $\hat W$ compactly contains $W$.
  
  In this setup we say that $g$ is \emph{$\hat{W}$-extensible}.
  When both connected components of $\hat V\setminus V$ have
  length at least $\delta |V|$ for some $\delta > 0$,
  we say that $g$ is \emph{$\delta$-extensible}.
\end{dfn}

\begin{lem}[Koebe Principle {\cite[Theorem~IV.1.2]{deMvS}, \cite{Mis:IHES}}]
  \label{lemExt}
  Suppose that $g \colon W \to V$ is a $\cC^2$ surjective diffeomorphism with $|Dg|^{-1/2}$ convex,
  and that $g$ is $\delta$-extensible. Then we have the \emph{distortion bound}
  $$
    \sup_{x,y \in W} \frac{Dg(x)}{Dg(y)} \leq \frac{(1+\delta)^2}{\delta^2}.
  $$
  In addition, there exists a constant $C$ depending only on $\delta$, such that
  for all $x,y \in W$,
  \[
    \bigl| \log |Dg(x)| - \log |Dg(y)| \bigr|
    \leq \frac{C}{|W|} |x-y|
    .
  \]
\end{lem}

Let us fix a constant $\Delta>1$ for which $\Delta$-extensible maps have distortion bounded by 2.

\begin{lem} \label{lemExtExp}
  Suppose that $g,W, \hat W, V, \hat V$ are as in Definition~\ref{defExt} and that,
  additionally, 
  each component of $\hat V \setminus V$ has length at least $10(1+\Delta)|V|$ and $|V|>|\hat W|$. 
  Then 
  $|Dg| > 5$ on $W$. 
\end{lem}

\begin{proof}
  There is an interval $V'\supset V$ with $|V'| = 10 |V|$. 
  Let $W' = g^{-1}(V')$. 
  Each component of $\hat V \setminus V'$ has length at least $10\Delta|V| = \Delta |V'|$, so $g \colon W' \to V'$ is $\Delta$-extensible. 
  By Lemma~\ref{lemExtExp},
  the distortion of $g$ is bounded by $2$ on $W'$. 
  The result then follows from the estimate $|V'| = 10|V| > 10 |W'|.$
\end{proof}

Suppose that $f \colon I \to I$ is a continuous map with $f(\partial I) \subset \partial I$.

\begin{dfn}
  We say that an interval $A \subset I$ is a \emph{pullback} of an interval $U \subset I$
  (under $f$), if $A$ is a connected component of $f^{-n}(U)$ for some $n \geq 0$.
\end{dfn}

\begin{dfn} 
  An open interval $U$ is called \emph{regularly returning} if 
  $f^n(\partial U) \cap U = \emptyset$ for all $n \geq 0$.
\end{dfn}
This property is widely used \cite{GSS:metric,Martens:Nice,PrzRL} to simplify the study of induced maps thanks to the following elementary property.

\begin{lem} \label{lemNice}
  If $U$ is regularly returning,
  then pullbacks of $U$ are either nested or disjoint, that is, if $A,B$ are pullbacks of
  $U$ and if $A \cap B \ne \emptyset$, then either $A \subset B$ or $B \subset A$.
\end{lem}

We shall use \emph{induced maps} of the form $F(x) = f^{\tau(x)}(x)$
in much of the paper, where $\tau$ is an \emph{inducing time}, 
defined on a disjoint union of open intervals, called
\emph{branches,} where $\tau$ is constant.
A branch is \emph{full} if its image equals the range of the induced map.

\emph{First entry} maps and \emph{first return} maps to a regularly returning interval $U$
are primary examples of induced maps.  The first entry time is
$$e(x) = \inf\{k\geq 0 : f^k(x) \in U\},$$ 
while the first return time is
$$r(x) = \inf\{k\geq 1 : f^k(x) \in U\} = 1 + e(f(x)).$$ 
The first entry map $x \mapsto f^{e(x)}(x)$ and the first return map
$x \mapsto f^{r(x)}(x)$ are defined on the sets 
$\{x \in I \colon e(x) < \infty\}$ and $\{x \in I \colon r(x) < \infty\}$ respectively.

Since $U$ is regularly returning, it follows from \rlem{Nice} that
if $W$ is a branch of the first entry or the first return map with the corresponding
inducing time $n_W$, then $f^{n_W}(\partial W) \subset \partial U$.

Henceforth, suppose that $\{f_t\}$ is a \mruf.
As a Misiurewicz map, $f_0$ enjoys strong expansion properties:

\begin{lem}[{\cite[Theorem~III.6.3]{deMvS}}]
  \label{lemMisEst}
  Given any sufficiently small neighbourhood $U$ of $0$, there exist $C\in(0,1)$ and
  $\lambda>1$ such that for each $x \in I$
  \begin{itemize}
    \item if $f_0^j(x) \notin U$ for $0 \leq j \leq k-1$, then 
      $$|Df_0^k(x)| \geq C\lambda^k;$$
    \item if $f_0^k(x) \in U$, then
      $$|Df_0^k(x)| \geq C\lambda^k.$$
  \end{itemize}
\end{lem}

The maps $f_t$ for $t \neq 0$ are not necessarily Misiurewicz,
and \rlem{MisEst} does not apply. Still, for $t$ small enough,
a similar statement holds:

\begin{lem}[{\cite[Theorem~III.6.4]{deMvS}}]  \label{lemMisEstF}
  There exists $C\in(0,1)$ and $\lambda>1$ such that,
  given any sufficiently small neighbourhood $U$ of $0$,
  the following holds for all sufficiently small $t$.
  \begin{itemize}
    \item If $f_t^j(x) \notin U$ for $0 \leq j \leq k-1$, then 
      $$|Df_t^k(x)| \geq C\lambda^k \inf_{0\leq j <k} |Df_t(f_t^j(x))|.$$
    \item If 
      $f_t^j(x) \notin U$ for $0 \leq j \leq k-1$ and $f_t^k(x) \in U$, then
      $$|Df_t^k(x)| \geq C\lambda^k.$$
  \end{itemize}
\end{lem}

Expansion entails a uniform distortion bound. 
\begin{lem}
  \label{lemDistExp}
  Let $U$ be a neighbourhood of $0$.
  There is a constant $C>1$ such that, for all $t$ small enough,  the following holds.
  If $W$ is an open interval such that 
  $f_t^k(W) \cap U = \emptyset$ for $0 \leq k < n$ and
   $x,y \in W$, 
  then 
  \[
    \bigl| \log |Df_t^n(x)| - \log |Df_t^n(y)| \bigr|
    \leq (\log C) |f_t^n(x) - f_t^n(y)|.
  \]
\end{lem}

\begin{proof}
  By \rlem{MisEstF}, there is a constant $C_0 > 0$
  (independent of $t$, $W$, $n$) such that, for all $x,y \in W$,
  \[
    \sum_{k=0}^{n-1} |f_t^k(x) - f_t^k(y)|
    \leq C_0 |f_t^n(x) - f_t^n(y)|
    .
  \]
  As $f_0$ has a unique critical point at $0$ and
  $f_t(x)$ is a $\cC^2$ function of $(x,t)$, $Df_t(x)$ is bounded away from 0 on $I\setminus U$ and $D^2f_t$ is bounded.
  Consequently, there exists a constant $C_1 > 0$, depending on $U$ but not on $t$,
  such that $|D (\log |D f_t|)| \leq C_1$ on $I \setminus U$.
  For $x,y \in W$, we deduce
  \begin{align*}
      \bigl| \log  |D f_t^{n}(y)| - \log |D f_t^{n}(x)| \bigr| &
    %\Bigl| \int_x^y  D (\log |D f_t^{n}|) (z) \, dz \Bigr|
     = \Bigl| \sum_{k=0}^{n-1} \int_{f_t^k(x)}^{f_t^k(y)} D (\log |D f_t|) (z) \, dz \Bigr| \\
     & \leq C_1
      \sum_{k=0}^{n-1} |f_t^k(x) - f_t^k(y)|
    \\ & \leq C_0 C_1
      \, |f_t^n(x) - f_t^n(y)|
    .
  \end{align*}
\end{proof}
The map $f_0$, being Misiurewicz, has an induced map with good properties.
\begin{lem}[{\cite[Proof of Lemma~V.3.2]{deMvS}}]
  \label{lemRegRet}
  For the map $f_0$, there is  an arbitrarily small regularly-returning open inverval $J$ containing $0$, disjoint from the post-critical orbit, for which $f(\partial J)$ is a (single) periodic point. Each branch of the first return map is mapped diffeomorphically \emph{onto} $J$. The complement in $J$ of the domain of the first return map has zero Lebesgue measure. There is a uniform distortion bound for all iterates of the first return map. 
\end{lem}

Let $\theta_0 > 0$ be small enough that for any
neighbourhood $U$ of $0$ contained in $(-\theta_0,\theta_0)$,
the conclusions of \rlem{MisEst} and \rlem{MisEstF} hold.
We further require that $\theta_0 < (10(1+\Delta))^{-1}$,
the latter constant as in \rlem{ExtExp}.

\begin{lem} \label{lemPrep}
    Let $J \subset (-\theta_0,\theta_0)$ be an interval given by \rlem{RegRet}.
  Periodic points of $f_0$ are dense in $J$. Preimages of any point in $J$ are dense in $J$ and hence in $I$. 
\end{lem}
\begin{proof} 
  Let $\phi \colon J \to J$ be the first return map to $J$ under the iterations of $f_0$.
  The union of branches of $\phi^n$ has full Lebesgue measure in $J$ for each $n$.
  Because of the uniform distortion and expansion bounds given by 
  Lemmas~\ref{lemRegRet} and~\ref{lemMisEst}, the maximal diameter of a 
  branch of $\phi^n$ tends to $0$ as $n\to \infty$. 
  
  Each branch $A$ of $\phi^n$ is mapped by $\phi^n$ diffeomorphically onto $J$.
  Assuming that $\partial A \cap \partial J = \emptyset$, there is a point $x \in A$
  such that $\phi^n(x) = x$. Thus all but at most two branches of $\phi^n$ contain a periodic
  point for $f$. It follows that periodic points are dense in $J$.
  
  For $x \in J$, each branch of $\phi^n$ contains a preimage of $x$, so the preimages are
  dense in $J$. Further, intervals $A \subset I \setminus J$ such that $f^n \colon A \to J$ is a
  diffeomorphism are dense in $I$ by expansion outside $J$ (see \rlem{MisEst}),
  so preimages of $x$ are dense in $I$.
\end{proof}

Let $\Lambda$ be a closed $f_0$-forward-invariant subset of $I$ such that $0\notin \Lambda$. 
We introduce the continuation of points in $\Lambda$ (see \cite[Lemma~3.1]{SanMMR}). 
\begin{lem}
There exist  an integer $N \geq 1$ and numbers $\rho, t_0, C >0$ such that the following holds for all $t \in [0,t_0]$. 
Given $x \in \Lambda$, there is a
unique point $x_t$ which satisfies $|x_t - x| \leq Ct$ and $|f_t^{Nj}(x_t) - f_0^{Nj}(x)| < \rho$ for all $j \geq 0$. 
The map $t \mapsto x_t$ is continuous.  
\end{lem}
\begin{proof}
    By \rlem{MisEstF}, for some $N\geq 1$ and $\rho >0$,
    $g_t := f_t^N$ satisfies $|Dg_t| > 2$ on a $\rho$-neighbourhood $B(\Lambda,\rho)$ of $\Lambda$ for all $t \in [0,t_1]$, for some $t_1>0$.
    Recall that $f_t(x)$ is a $\cC^2$ function of $(x,t)$.
    Choose $C>1$ such that $|g_t(x) - g_0(x)| \leq Ct$ for all $x$ and all $t \in [0,t_1]$. 

    We now apply the Implicit Function Theorem.
    There exists $t_0 \in (0,\min(t_1, \frac{\rho}C))$ such that,
    for each $t\in [0,t_0]$ 
     and $x \in \Lambda$, 
    there is a unique $y$ in the same connected component of $B(\Lambda, \rho)$ as $x$ which satisfies
    $g_0(x) = g_t(y)$.
    Moreover,
    \[
        |x-y|
        \leq \frac{1}{2} \bigl| g_t(x) - g_t(y) \bigr|
        = \frac{1}{2} \bigl| g_t(x) - g_0(x) \bigr|
        \leq \frac{Ct}{2}
        < \frac{\rho}{2}
        .
    \]
    
    Fix such $x, t$ and $y$. 
    If $z \in B(g_0(x), Ct)$, there is a unique $y'$ in the same connected
    component as $x$ of $B(\Lambda, \rho)$ with $z= g_t(y')$
    and $y'$ satisfies
    \[
        |y' - x|
        \leq |y' - y| + |y - x|
        \leq \frac12 | g_t(y') - g_t(y) | + Ct/2
        \leq Ct
        .
    \]
    Inductively, for $n\geq 0$ we obtain points $y_n = y_n(t)$ such that
    $g_t^n(y_n) = g_0^n(x)$ and $|g_t^j(y_n) - g_0^j(x)| \leq Ct$ for
    $j=0,\ldots,n$ and $y_{n+1} \subset B(y_{n}, 2^{-n})$.

    In particular, $(y_n)_n$ is a Cauchy sequence whose limit we denote by $x_t$.
    The point $x_t$  satisfies $|x_t - x| \leq Ct$ and $|g_t^j(x_t) - g_0^j(x)| < \rho$ for all $j \geq 0$.
    Continuous dependence of $x_t$ on $t$ follows from continous dependence of $y_n(t)$ on $t$. 
\end{proof}

In particular, for each $x \in \Lambda$, we obtain a map $t \mapsto x_t$ with the same Lipschitz constant $C$. 
Combining them generates a map $t \mapsto \Lambda_t$. 
Note that if $x$ is preperiodic for $f_0$, then from uniqueness it follows that $x_t$ is preperiodic for $f_t$.

\begin{dfn}[Continuation] \label{defcont}
    The map $t \mapsto x_t$ as above (or the point $x_t$) is called the \emph{continuation} of $x = x_0$.
    $\Lambda_t$ is called the \emph{continuation} of $\Lambda = \Lambda_0$. 
\end{dfn}

\begin{lem} \label{lemUt0}
  Let  $\theta \in (0,\theta_0)$. For sufficiently small $t$, there exist
  open intervals $U_0$, $U_1$,  such that
  \begin{enumerate}[label=(\alph*)]
    \item $0 \in  U_1 \subset U_0 \subset (-\theta,\theta)$;
    \item for each $j$, the boundary $\partial U_j$ varies continuously with $t$,
      and $f_t(\partial U_j)$ is a single point, preperiodic  with respect to $f_t$;
    \item\label{lemUt0:bdd}
          $f_t^k(\partial U_j) \notin U_0$ for all $k \geq 1$ and $j=0,1$;
    \item $|U_1| \leq \theta \dist(U_1, \partial U_0)$. 
  \end{enumerate}
\end{lem}

\begin{proof}
  Suppose first that $t = 0$.
  Let $J \subset (-\theta/2,\theta/2)$ be given by \rlem{Prep}, and set $U_0 = J$.
  Recall that $f_0(\partial U_0)$ is a single periodic point whose 
  orbit under $f_0$ is disjoint from $U_0$.

  Let $F \colon U_0 \to U_0$ denote the first return map to $U_0$
  under $f_0$. Branches of $F$ accumulate on $0$,
  since $0$ never returns, and boundary points of branches get
  mapped by the corresponding iterate of $f_0$ to $\partial U_0$. 
  Hence there are preperiodic points, arbitrarily close to $0$,
  which never  return to $U_0$. Choose one, $p<0$, such that $p$ and
  its symmetric point $p_*$ (in the sense $f_0(p)=f_0(p_*)$) lie in $U_0$
  and such that
  $$
    |p_*-p| < \theta \dist\bigl((p,p_*), \partial U_0\bigr)/2,
  $$
  and set $U_1 = (p,p_*)$. 

  The boundaries of $U_j$, $j=0,1$, consist of preperiodic points whose
  forward orbits do not include $0$, hence they admit continuations,
  giving the sets $U_j$ with the required properties for small enough $t$. 
\end{proof}

\begin{lem} \label{lemER}
  Let $U_j$ denote the intervals from \rlem{Ut0}.
  Let
    \begin{align*}
      E_n & = \{x \in I \colon f_t^k(x) \notin U_1 \, \text{ for all } \, k = 0,1,2,\ldots, n\}, \\
      R_n & = \{x \in I \colon f_t^k(x) \notin U_1 \, \text{ for all } \, k = 1,2,\ldots, n\}.
    \end{align*}
    For $t$ small enough, there are constants $\alpha, C>0$ such that
    $$
      m(E_n) < C e^{-\alpha n}
      \qquad \text{and} \qquad
      m(R_n) < C e^{-\alpha n}
    $$
    for all $n \geq 0$.
\end{lem}

\begin{proof}
  Choose a neighbourhood of $0$ contained in $U_1$ for all small $t$ and obtain
  a distortion bound $C'>1$ from \rlem{DistExp}. 
  Let us drop the dependence on $t$ from notation, where appropriate.

  Note that $E_n$ is a finite union of closed intervals
  and $E_{n+1} \subset E_{n}$. 
  Let $A$ be a connected component of $E_n$. Then $f^n$ is monotone on $A$ and
  the boundary points of the interval $f^n(A)$ are distinct elements of the preperiodic forward
  orbit of $\partial U_1$. Therefore, $|f^n(A)| > \kappa_1$, where $\kappa_1>0$
  is  independent of $A$, $n$ and $t$ (for $t$ small enough).
  Hence there exists a number $N$ (independent of $A$, $n$ and $t$) such that
  $f^{n+k} (A) \cap U_0 \neq \emptyset$ for some (minimal) $k \leq N$. In fact,
  $U_0 \subset f^{n+k} (A)$, because the boundary points of $f^n(A)$
  never return to $U_0$ under iteration of $f$. Also,
  $f^{n+k} \colon A \to f^{n+k}(A)$ is a diffeomorphism and
  $A \setminus E_{n+k}$ is a subinterval of $A$ such that $f^{n+k}(A \setminus E_{n+k}) = U_1$.
  The distortion of $f^{n+k}$ is bounded by $C'$ on $A$, by \rlem{DistExp}. Consequently
  $$
  \frac{m(A \setminus E_{n+k})}{m(A)} \geq C'^{-1} \frac{|U_1|}{|I|}.
  $$
  
  Hence there exists $\gamma \in (0,1)$, independent of $A,n,t$, for which
  \[
      m(A \cap E_{n+N}) \leq m(A \cap E_{n+k}) \leq \gamma m(A)
    .
  \]
  Summing over all connected
  components of $E_n$, we obtain $m(E_{n+N}) \leq \gamma m(E_n)$.
  The result for $m(E_n)$ follows by induction. 
  Since $f(R_n) \subset E_{n-1}$ and $f$ has a quadratic critical point, $m(R_n) \lll \sqrt{m(E_{n-1})}$, so
  we also obtain the result for $m(R_n)$.
\end{proof}

Denote $f_t^{n+1}(0)$ by $\xi_n(t)$. The proof of the following lemma is
based on \cite{Tsu}; the ideas go back at least to \cite{BenCar}.

\begin{lem} \label{lemGamman}
  If $\{f_t\}$ is transveral, there exist $r_0>0$, $m_0 \geq 1$ 
  and a sequence of positive numbers $\gamma_n, n \geq m_0$ with 
  \begin{enumerate}[label=(\alph*)]
    \item \label{lemgn-a} 
      $\gamma_n/\gamma_{n+1} \mc 1, \quad \lim_{n \to \infty} \gamma_n = 0$; 
    \item \label{lemgn-b}
      $\gamma_n^{-1} \mc |D\xi_n(0)| \mc |Df_0^n(f_0(0))|;$
    \item \label{lemgn-c}
      $|\xi_n(\gamma_n) - \xi_n(0)| \geq r_0;$
      %(with multiplicative constant independent of $t \leq \gamma_n$)
    \item \label{lemgn-d}
      for all $m_0 \leq k \leq n$, the map $\xi_k$ is monotone on
      $[0,\gamma_n]$ and has a distortion bound
      \[
        \displaystyle 
        \Bigl|\log \frac{|D \xi_k(s)|}{|D \xi_k(t)|} \Bigr|
        \leq 1
        \quad \text{for all } \ s,t \in [0,\gamma_n]
        ;
      \]
    \item \label{lemgn-e}
      \[
        \displaystyle 
        \Bigl|\log \frac{|Df_0^n(f_0(0))|}{|Df_t^n(f_t(0))|}
        \Bigr| \leq 1
        \quad \text{for all } \ t \in [0,\gamma_n]
        .
      \]
  \end{enumerate}
\end{lem}

\begin{proof}
    Recall from \rlem{MisEst} that
   $|Df_0^k(f_0(0))| \geq C_0\lambda^k$. 
  We use Tsujii \cite{Tsu} and only treat large $n$.
  From \cite[Equation~3.3]{Tsu},
  \[
    |Df^n_0(f_0(0))|^{-1} \mc a^+(f_0(0),n; 0)
    ,
  \]
  where 
  \(
    a^+(x,n; t) = \Bigl(
    4 e \kappa_1^2
    \sum_{j=0}^{n-1} \frac{|Df_t^j(x)|}{Df_t(f_t^j(x))|}
    \Bigr)^{-1}
  \)
  and $\kappa_1 > 1$.

  We choose $\gamma_n$ equal to $\gamma^{(\mu)}(0,n)$
  in \cite[Section~5]{Tsu}. By \cite[Lemma~5.2]{Tsu} and the
  preceding Remarks with $t=0$,
  \begin{itemize}
    \item $|D\xi_n(0)| \mc |Df_0^n(f_0(0))|$;
%    \item 
%                on $[0,\gamma_n]$,  
%        $\xi_k$ 
%        for $m_0 \leq k \leq n$ has distortion bounded by $e$;
    \item \(\gamma_n < |D\xi_n(0)|^{-1}\);
    \item \(\gamma_n \ggg a^+(f_0(0), n; 0)\).
  \end{itemize}
  Hence we obtain~\ref{lemgn-b} which in turn implies~\ref{lemgn-a}.
  
  Bounds~\ref{lemgn-d} and~\ref{lemgn-e} correspond to
  \cite[$\Gamma1$ and $\Gamma2$]{Tsu}.
  Finally,~\ref{lemgn-c} follows from $\gamma_n \mc |D\xi_n(0)|^{-1}$ and~\ref{lemgn-d}.
\end{proof}

\section{First return maps} \label{sectSetUp}

We continue to suppose that $\{f_t\}$ is a \mruf.
Let $\Lambda_0$ be the closure of the post-critical orbit of $f_0$.
Let $\Lambda_t$ be its continuation, see \rdef{cont}.

Where appropriate, 
 we shall suppress the dependence on $t$ from notation for better legibility.

Given the intervals $U_j$, as in \rlem{Ut0}, we denote by
$\phi_j \colon U_j \to U_j$  the first return map
under iteration by $f_t$, and by $\psi_j \colon I \to U_j$
the first entry map.

\begin{lem} \label{lemBigDer}
  There are constants
  $C>1$, $\theta_1 \in (0,\theta_0)$ such that for $\theta \in (0,\theta_1)$, if $U_j$, $j=0,1$, are given by \rlem{Ut0}, if $t$ is small and if $x \in U_j$ with $|x|>Ct$, then
  \[
    |D\phi_j(x)| \geq 1000
    .
  \]
\end{lem}

\begin{proof}
  Let $\delta_0 = \frac{1}{4} \dist (\Lambda_0, 0)$.
  Set $y_0 = f_0(0) \in \Lambda_0$ and
  let $y_t$ denote the continuation of $y_0$. 
  Suppose that $x$ is small and $f_t(x) \ne y_t$.
  Then
  \begin{align*}
      |f_t(x) - y_t| 
      & \leq |f_t(x) - f_t(0)| + |f_t(0) - y_0| + |y_0 - y_t| 
      \\ & \lll x^2 + t.
  \end{align*}
  Let $W = (f_t(x), y_t)$ and set
  $$
    n = \inf\{
      k \geq 0 \colon |f_t^{k}(W)| \geq \delta_0
    \}.
  $$
    
  As $y_t$ is in the $f_t$-invariant set $\Lambda_t$, 
  \begin{equation}
    \label{eq:anhho}
    f_t^k(W) \cap (-\delta_0, \delta_0) = \emptyset
  \end{equation}
  for all $0 \leq k < n$.
  By \rlem{MisEstF}, $|Df_t^k|\geq C \lambda^k$ on $W$ for
  some $C' > 0$ and $\lambda > 1$ independent of $x$ and $t$,  for all $0\leq k < n$. 
  Hence $n$ is finite. 
  By \rlem{DistExp}, $f_t^n$ has bounded distortion on $W$,
  independent of $x$ and $t$.
  Hence, 
  \[
    |Df_t^n(f_t(x))| 
    \ggg \frac{1}{|f_t(x)-y_t|}
  \]
  and
  \[
    |Df_t^{n+1}(x)| 
    \ggg \frac{|x|}{|f_t(x) - y_t|}
    \ggg \frac{|x|}{x^2 + t}
    .
  \]
  By \rlem{MisEstF}, there is $C'>0$ such that the first entry map to any
  sufficiently small neighbourhood $U$ of $0$ has derivative at least $C'$.
  Further, if $U \subset (-\delta_0, \delta_0)$, then by~\eqref{eq:anhho},
  the first return time of $x$ to $U$ under $f_t$ is at least $n$.
  Hence if $\phi$ is the first return map to $U$, then, provided that $U$ and $t$ are
  small enough,
  \[
    |D \phi(x)|
    \ggg \frac{|x|}{x^2+t}
    ,
  \]
  with the implied constant independent of $U$ or $t$.

  One can therefore choose $C >1$ and $\theta_1 \in (0,\theta_0)$ so that, if
  $U \subset (-\theta_1, \theta_1)$ and $x \in U$, $|x| \geq Ct$, then  
  $|D\phi(x)| \geq 1000$.

  Now if $\theta \in (0,\theta_1)$ and $U_j$ are given by \rlem{Ut0} with first return maps $\phi_j$, then the above estimates imply that
  $|D\phi_j(x)| \geq 1000$ provided $t$ is small enough, $|x| \geq Ct$ and $x \in U_j$. 
\end{proof}

\begin{lem} \label{lemUt1}
  If $W$ is a branch of $\psi_1$, then there is an open interval 
  $\hat W$, with $W \subset \hat W$, mapped diffeomorphically by
  $f_t^n$ onto $U_0$, where $\psi_1 = f_t^n$ on $W$. 
\end{lem}

\begin{proof}
  If $W \ni 0$, then $n=0$, $f_t^n$ is the identity map  and the claim is
  trivial indeed.
  Suppose $0 \notin W$. Let $\hat W \supset W$ be the maximal open interval
  with $f_t^n(\hat W) \subset U_0$.
  Since 
  \(
    f_t^k(\partial U_1) \cap U_0 = \emptyset
  \)
  for $k \geq 1$,
  \[
    f_t^j(\hat W) \cap \partial U_1 = \emptyset
    \qquad \text{for all } 0 \leq j < n
    .
  \]
  Since $n$ is the first entry time on $W$,
  \[
    f_t^j(\hat W) \cap U_1 = \emptyset
    \qquad \text{for all } 0 \leq j < n
    .
  \]
  Hence $f_t^n$ has no critical point in a neighbourhood
  of $\overline {\hat W}$, and maximality gives surjectivity. 
\end{proof}

If $\phi_1$ has a critical point, it is unique and equal to $0$.
Otherwise, $\phi_1$ is not defined at $0$. A branch of $\phi_1$
containing $0$ is called \emph{central}. 

\begin{lem} \label{lemUt2}
  Suppose that either $0$ never returns to $U_0$ or the first return of $0$ to $U_0$ lies
  in $U_1$. Let $W$ be a non-central branch of $\phi_1$.
  Then there is an open interval $\hat W$, with $W \subset \hat W \subset U_1$,
  mapped diffeomorphically by $f_t^n$ onto $U_0$, where $\phi_1 = f_t^n$ on $W$.
  In case $\phi_1$ has a central branch, $\hat W$ is disjoint from it.
  On the non-central branches, $|D\phi_1| \geq 5$. 
\end{lem}

\begin{proof}
  As in the proof of \rlem{Ut1},
  let $\hat W \supset W$ be the maximal open interval with $f_t^n(\hat W) \subset U_0$.
  Then $f_t^j(\hat W) \cap \partial U_1 = \emptyset$ for $0 \leq j < n$,
  in particular, $\hat W \subset U_1$.
  Since $n$ is the first return time on $W$, 
  \begin{equation}
    \label{eq:nss}
    f_t^j(\hat W) \cap U_1 = \emptyset
    \qquad \text{for } 1 \leq j < n
    .
  \end{equation}
  Therefore $0$ is the only possible critical point of $f^n_t$ on $\hat W$.

  Next we show that $0 \not \in {\hat W}$. Indeed, suppose that $0 \in \hat W$.
  Then by \eqref{eq:nss} and by the first return hypothesis, $f_t^k(0) \notin U_0$
  for $1 \leq k < n$, thus $n$ is the first return time of $0$ to $U_0^t$.
  Again by the first return hypothesis, $f_t^n(0) \in U_1$. Since $0$ is the only critical point
  of $f_t^n$ on $\hat{W}$, all points between $0$ and $W$ get mapped by $f_t^n$ into $U_1$, so 
  % and the first return times to $U_1$ are the same, % <--- this was a bit confusing
  $0 \in W$, contradicting our assumption that $W$ is non-central.

  Since $\hat W$ is the \emph{maximal} open interval with $f_t^n(\hat W) \subset U_0$
  and $f_t^n$ has no critical points on $\hat W$, it follows that
  $f_t^n(\hat W) = U_0$.

  Now let us show that in case $\phi_{1,t}$ has a central branch,
  $\hat W$ is disjoint from it.
  Suppose that $Z$ is the central branch with return time $n_0$
  and that $\hat W \cap Z \neq \emptyset$.
  Since $0 \in Z$ and $0 \not \in \hat W$, it follows that 
  there is $x \in \partial \hat W \cap Z$. Then $f_t^n(x) \in \partial U_0$,
  so $f_t^k(x) \not \in U_0$ for all $k \geq n$, thus
  $n_0 < n$. Hence, $f_t^n(\partial Z) \not \in U_0$,
  so $\partial Z \cap \hat{W} = \emptyset$. It follows that
  $Z$ contains $\hat W$ and $n_0 = n$, which contradicts 
  $n_0 < n$.

  Since $\theta < \theta_0 < \frac{1}{10(1+\Delta)}$ and
  $|U_1| < \theta \dist(U_1, \partial U_0)$,
  the derivative estimate follows  from \rlem{ExtExp}.
\end{proof}

\section{Breakdown of statistical stability} \label{sectDiscon}

In this section, we suppose that our \mruf{} is transversal and  prove Theorem~\ref{thmDiscon}. We again let $\Lambda_0$ denote the closure of the post-critical orbit of $f_0$ and $\Lambda_t$ the continuation of $\Lambda_0$. The absolutely continuous invariant probability measure for $f_0$ is $\mu_0$. 

\begin{lem}
  Given any $\eps>0$, there is a neighbourhood $W_\Lambda$ of the post-critical
  set $\Lambda_0$ of $f_0$ and a $\cC^\infty$ observable $\pot$ with $\pot \geq 0$ for which
  $$\pot(x)=1$$ for all $x \in W_\Lambda$ and for which
  $$\int \pot \, d\mu_0 < \eps.$$
\end{lem}
\begin{proof}
 By \rlem{ER}, 
$m(\Lambda_0)=0$.
As $\Lambda_0$ is compact and $\mu_0$ is absolutely continuous, Urysohn's Lemma provides a continuous function which is 1 on $\Lambda_0$ and $0$ on a closed set of $\mu_0$-measure $1-\eps$. Perturbing this function, the result follows. 
\end{proof}
\iffalse
\begin{proof}
  Since $\mu_0$ is absolutely continuous, there exists a cover of $\Lambda_0$ by open balls whose union has
  $\mu_0$-measure less than $\eps/2$. Since $\Lambda_0$ is compact, one can extract
  a finite subcover of balls $B(x_i, r_i)$ with center $x_i$ and radius $r_i$
  with $\cup_i B(x_i, r_i) = W_\Lambda$. For $\delta>0$ small enough,
  the union of $B(x_i, r_i +\delta)$ will have measure at most $\eps$, and there is a 
  smooth function $\pot$ taking values in $[0,1]$ with $\varphi = 1$ on
  $W_\Lambda$ and zero on $I \setminus \cup_iB(x_i, r_i+\delta).$
\end{proof}
\fi

Showing \rthm{Discon} therefore reduces to proving the following proposition, whose proof takes the rest of this section.

\begin{prop} \label{propDisconReduc}
    Let $a>0$.
  There exists $\alpha_0>0$ such that, for any neighbourhood $W_\Lambda$ of
  $\Lambda_0$ with the characteristic function $1_{W_\Lambda}$, 
  \[
    \limsup_{t \to 0^+} \int_{I} \sbar_{t, \lfloor at^{-1} \rfloor} 1_{W_\Lambda} \, dm
    \geq \alpha_0.
  \]
\end{prop}

Our strategy is to construct a sequence $t_n$ with $\lim_{n \to \infty} t_n = 0$
such that: the maps $f_{t_n}$ have $0$ as a super-attracting periodic point; 
most of the \emph{immediate basin of attraction} 
of the corresponding periodic orbit is contained in a small
neighbourhood of $\Lambda_0$; a definite proportion of all points 
in $I$ enter the immediate basin in fewer than
$\lfloor t^{-1} \rfloor / 2$ iterates. 

\begin{dfn}
The \emph{immediate basin of attraction} of a periodic point is the union of the connected components of the basin of attraction which contain points of the periodic orbit. 
\end{dfn}

Let $r_0$, $m_0$, $(\gamma_n)_{n\geq m_0}$ be as in \rlem{Gamman}.
Let $\theta_1>0$ be given by \rlem{BigDer}.

\begin{lem} \label{lemafter0} 
    There are $N\geq1$, $\theta \in (0,\theta_1)$ and a sequence of parameters $t_n >0$ such that
  \begin{enumerate}[label=(\alph*)]
    \item\label{lemafter0:a} 
      \(
        t_n 
        \mc \gamma_n 
        \mc |Df^{n}_{t_n}(f_{t_n}(0))|^{-1}
         \mc |Df^{n}_{0}(f_{0}(0))|^{-1}
      \);
    \item\label{lemafter0:b}
        for some $p_n \in [n, n+N],$ $f_{t_n}^{p_n} (0) = 0$ and $f_{t_n}^k(0) \notin (-\theta,\theta)$ for $0< k<p_n$.
       %$\xi_k([0,t_n]) \cap (-\theta, \theta) = \emptyset$
      %for  $k \leq n$.
  \end{enumerate}
\end{lem}

\begin{proof}
    Recall that by \rlem{Gamman}, for $m_0 \leq k \leq n$, the map $\xi_k$ is monotone on $[0, \gamma_n]$
    and has universally bounded distortion. Thus $|\xi_k([0, \eps \gamma_n])| \lll \eps$ for $\eps > 0$.
    For $k < m_0$, we bound $|\xi_k([0, \eps \gamma_n])| \leq \eps \sup_{j < m_0} \sup_t |D \xi_j(t)|$.
    Overall,
    \[
        |\xi_k([0, \eps \gamma_n])|
        \lll \eps
        \quad \text{for all } k \leq n
        .
    \]
    We choose $\eps_0$ small enough so that
    $$\dist(\xi_k([0,\eps_0\gamma_n]), 0) > \dist(\Lambda_0, 0)/2 \quad \mbox{for all } k\leq n.$$
    By \rlem{Gamman}, $|\xi_n([0,\gamma_n])| \geq r_0$; since $\xi_n$ has bounded distortion,
    there is an $\eps_1>0$ for which 
    $|\xi_n([0, \eps_0\gamma_n])| > \eps_1$ for all large $n$. 
    Note that $\eps_1 < \dist(\Lambda_0, 0)/2$. 
Fix $N$ large so that, setting
\[
  Q_t = \bigcup_{k=1}^{N-1} f_t^{-k}(0),
\]
$Q_0$ is $\eps_1/3$-dense in $I$, see \rlem{Prep}.
For $t$ small, $Q_t$  is  $\eps_1/2$-dense.
There is $\theta \in (0, \theta_1)$ for which $Q_t \cap (-\theta, \theta) = \emptyset$
for small $t$. Moreover $\dist(Q_t, \Lambda_0) \mc 1$.

Define
\[
  t_n = \min \{
    t \in [0,\gamma_n] \colon
     \xi_n(t) \in Q_t 
  \}
  .
\]
  By construction, $0 < t_n < \eps_0\gamma_n$ and \ref{lemafter0:b} holds. 
  By \rlem{Gamman}, $\xi_{n}$ acts on $[0, \gamma_n]$
  as a diffeomorphism with bounded distortion.
  It follows from $\dist(\Lambda_0, Q_t) \mc 1$ 
  that $|\xi_n(t_n) -\xi_n(0)| \mc 1$.
  Thus $t_n \mc \gamma_n$; the remaining relations in
  \ref{lemafter0:a} follow from \rlem{Gamman}.
\end{proof}

We now work with the fixed map $f = f_{t_n}$, where 
$n$ is as large as necessary. Write $p = p_n$ for the period of $0$.
Let the intervals $U_j$ be given by \rlem{Ut0} for $\theta$ from \rlem{after0}.
Let $\phi_1$ denote the first return map to $U_1$.
An example graph of $\phi_1$ is shown on Figure~\ref{fig:m}.

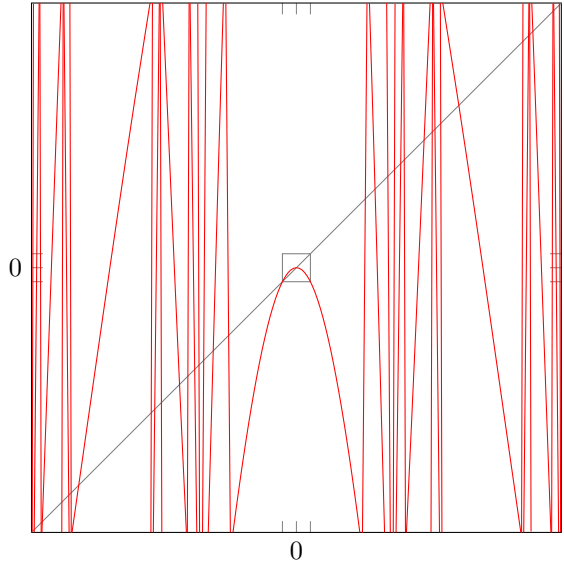
\begin{figure}[ht!]
\begin{tikzpicture}[
    declare function={
      logistic(\x) = 3.9602701272211527 * \x * (1.0 - \x);
      logisticc(\x)                        = logistic(logistic(\x));
      logisticcc(\x)                       = logistic(logisticc(\x));
      logisticccc(\x)                      = logistic(logisticcc(\x));
      logisticcccc(\x)                     = logistic(logisticccc(\x));
      logisticccccc(\x)                    = logistic(logisticcccc(\x));
      logisticcccccc(\x)                   = logistic(logisticccccc(\x));
      logisticccccccc(\x)                  = logistic(logisticcccccc(\x));
      logisticcccccccc(\x)                 = logistic(logisticccccccc(\x));
    }
  ]
  \begin{axis}[
    width=10cm,
    axis equal image,
    xmin=0.378448, xmax=0.621552,
    ymin=0.378448, ymax=0.621552,
    y dir=reverse,
    % ticks
    xtick={0.49358, 0.5, 0.50642},
    xticklabels={,$0$,},
    ytick={0.49358, 0.5, 0.50642},
    yticklabels={,$0$,},
    %grid=major,
    % extra x ticks
    % y ticks
    %ymajorticks=false,
    %ytick={0.0, 1.0},
    %yticklabels={},
	]
    
    % diagonal and box
    \addplot[color=gray,domain=0.378448:0.621552]{1-x};
    \addplot[color=gray] coordinates { (0.50642, 0.50642) (0.50642, 1-0.50642) 
    (1-0.50642, 1-0.50642) (1-0.50642, 0.50642) (0.50642, 0.50642) };

    % branches
    \addplot[color=red,smooth,samples=6,domain=0.397029:0.432842]{logisticcc(x)};
    \addplot[color=red,smooth,samples=6,domain=0.567158:0.602971]{logisticcc(x)};

    \addplot[color=red,domain=0.470893:0.500000]{logisticccc(x)};
    \addplot[color=red,domain=0.500000:0.529107]{logisticccc(x)};

    \addplot[color=red,smooth,samples=6,domain=0.438652:0.449519]{logisticcccc(x)};
    \addplot[color=red,smooth,samples=6,domain=0.550481:0.561348]{logisticcccc(x)};
    \addplot[color=red,smooth,samples=6,domain=0.383229:0.391996]{logisticcccc(x)};
    \addplot[color=red,smooth,samples=6,domain=0.608004:0.616771]{logisticcccc(x)};

    \addplot[color=red,smooth,samples=6,domain=0.459548:0.466489]{logisticccccc(x)};
    \addplot[color=red,smooth,samples=6,domain=0.533511:0.540452]{logisticccccc(x)};

    \addplot[color=red,smooth,samples=6,domain=0.434350:0.437090]{logisticcccccc(x)};
    \addplot[color=red,smooth,samples=6,domain=0.562910:0.565650]{logisticcccccc(x)};
    \addplot[color=red,smooth,samples=6,domain=0.393326:0.395701]{logisticcccccc(x)};
    \addplot[color=red,smooth,samples=6,domain=0.604299:0.606674]{logisticcccccc(x)};
    \addplot[color=red,smooth,samples=6,domain=0.451215:0.454282]{logisticcccccc(x)};
    \addplot[color=red,smooth,samples=6,domain=0.545718:0.548785]{logisticcccccc(x)};
    \addplot[color=red,smooth,samples=6,domain=0.379689:0.381945]{logisticcccccc(x)};
    \addplot[color=red,smooth,samples=6,domain=0.618055:0.620311]{logisticcccccc(x)};

    \addplot[color=red,smooth,samples=6,domain=0.467616:0.469691]{logisticccccccc(x)};
    \addplot[color=red,smooth,samples=6,domain=0.530309:0.532384]{logisticccccccc(x)};
    \addplot[color=red,smooth,samples=6,domain=0.456922:0.458587]{logisticccccccc(x)};
    \addplot[color=red,smooth,samples=6,domain=0.541413:0.543078]{logisticccccccc(x)};

    \addplot[color=red,smooth,samples=6,domain=0.437504:0.438240]{logisticcccccccc(x)};
    \addplot[color=red,smooth,samples=6,domain=0.561760:0.562496]{logisticcccccccc(x)};
    \addplot[color=red,smooth,samples=6,domain=0.392345:0.392972]{logisticcccccccc(x)};
    \addplot[color=red,smooth,samples=6,domain=0.607028:0.607655]{logisticcccccccc(x)};
    \addplot[color=red,smooth,samples=6,domain=0.449962:0.450762]{logisticcccccccc(x)};
    \addplot[color=red,smooth,samples=6,domain=0.549238:0.550038]{logisticcccccccc(x)};
    \addplot[color=red,smooth,samples=6,domain=0.382286:0.382891]{logisticcccccccc(x)};
    \addplot[color=red,smooth,samples=6,domain=0.617109:0.617714]{logisticcccccccc(x)};
    \addplot[color=red,smooth,samples=6,domain=0.433234:0.433946]{logisticcccccccc(x)};
    \addplot[color=red,smooth,samples=6,domain=0.566054:0.566766]{logisticcccccccc(x)};
    \addplot[color=red,smooth,samples=6,domain=0.396055:0.396682]{logisticcccccccc(x)};
    \addplot[color=red,smooth,samples=6,domain=0.603318:0.603945]{logisticcccccccc(x)};
    \addplot[color=red,smooth,samples=6,domain=0.454744:0.455564]{logisticcccccccc(x)};
    \addplot[color=red,smooth,samples=6,domain=0.544436:0.545256]{logisticcccccccc(x)};
    \addplot[color=red,smooth,samples=6,domain=0.378771:0.379356]{logisticcccccccc(x)};
    \addplot[color=red,smooth,samples=6,domain=0.620644:0.621229]{logisticcccccccc(x)};
  \end{axis}
\end{tikzpicture}
\caption{
  Graph of $\phi_1 \colon U_1 \to U_1$ when $0$ is a periodic point.
  Between every two branches there are countably many other branches;
  $\phi_1$ is uniformly expanding outside the small invariant interval % $V_n$
  in the middle.
}
\label{fig:m}
%\label{fig:m}
\end{figure}

\rlem{after0} guarantees that the first return of $0$ under $f$ to $U_0$ 
is $0 \in U_1$, thus by \rlem{Ut2}, $\phi_1$ restricted to $U_1$ 
 has a unimodal central branch which we denote by $Z$;
all other branches are full with a uniform distortion bound.
On $Z$, $\phi_1 = f^p$. 
We denote  the immediate
basin of attraction (with respect to $\phi_1$) of $0$ by 
$V$. As $\phi_1(0) = 0$, $V$ is an interval.

\begin{lem} \label{lemWlambda}
      Given any 
      neighbourhood $W_\Lambda$ of $\Lambda_0$ and $\eps >0$,  
      the following holds for all $n$ large enough.
      For all $x \in V$ and $k \geq 1$, 
      the Birkhoff average of the characteristic function $1_{W_\Lambda}$ of $W_\Lambda$ satisfies
      $$
      \sbar_{t_n, k} \, 1_{W_\Lambda}(f(x)) \geq 1 - \eps
        .
      $$
\end{lem}  
\begin{proof}
   Note that $\phi_1(V) \subset V$
  and recall that the first return of $0$ to $U_0$ is at time $p$
  with $n+1 \leq p \leq n + N$.
  Given $\eps >0$ we shall show, 
  for large $n$ and $j \leq (1-\eps) n$,  that $f^j(V)$  and 
  $\dist(f^j(V), \Lambda_0)$ are sufficiently small to guarantee
  that $f^j(V) \subset W_\Lambda$.
  Since $f^{p\ell}(V) \subset V$ for each $\ell \geq 0$, this implies
  that $f^j(V) \subset W_\Lambda$ for all $p\ell + 1 \leq j < p\ell + (1-\eps) n$.
  From this, the Birkhoff estimate follows. 

  For $j = 1, \ldots, n$, $f^j(V) \cap U_0 = \emptyset$. 
  \rlem{MisEstF} implies that $|f^{j}(V)|$ is exponentially small in $n-j$.
  By \rlem{MisEst},  
  $|Df_0^{n-j}(f_0^k(0))| \ggg \lambda^{n-j}$. 
  With the estimates of  \rlem{Gamman}, one deduces that 
  $
    \dist(f^j(0), \Lambda_0) 
    = \dist(\xi_{j-1}(t_n),\Lambda_0)
  $ 
  is exponentially small in $n-j$. Thus so is
  $\dist(f^j(V), \Lambda_0)$. The proof 
  is complete.
  \end{proof}
We establish properties of $\phi_1$ on $Z$.
\begin{lem} \
  \label{lemphi}
  \begin{enumerate}[label=(\alph*)]
    \item\label{lemphi:a} $|Z| \mc t_n^{1/2}$ and $|V| \mc t_n$;
    \item\label{lemphi:b} there exists $\eta >0$, independent of $n$,
        such that $|D\phi_1| > e^{n\eta}$ on $Z \setminus \phi_1^{-1}(Z)$;
    \item\label{lemphi:c} $\log |D\phi_1| > 1/2$ on $U_1 \setminus V$.
  \end{enumerate}
\end{lem}  

\begin{proof}
  Since $n+1 \leq p \leq n + N$, $|Df^n(f(0))| \mc |Df^{p-1}(f(0))|$. 
  Let $\psi_1$ be the first entry map to $U_1$.
  Its branches have bounded distortion (\rlem{Ut1}), so
  \rlem{after0} entails that
  \[
    |\psi_1'(f(0))|
    \mc |D f^{p-1}(f(0))|
    \mc |D f^n(f(0))|
    \mc t_n^{-1}
    .
  \]
  The interval $V$ is a neighbourhood of the non-degenerate
  critical point, $|f(V)| \mc |V|^2$. At the same time, $|\phi_1(V)| \mc |V|$
  (see Figure~\ref{fig:m}: $Z$ is the domain of the central branch,
  and $V$ is the small invariant interval in the middle).
  Observe that $f^{p} = \psi_1 \circ f$ on $Z$. 
  Hence 
  $$
    |V| \mc |\psi_1'(f(0))|^{-1} \mc t_n
    .
  $$
  Meanwhile, $|\phi_1(Z)| \mc 1$, so
  $|f(Z)| \mc t_n$ and $|Z| \mc t_n^{1/2}$.
  This proves~\ref{lemphi:a}.

  Let $I_1 = Z \cap \phi_1^{-1}(Z)$. By a similar argument, 
  $$ 
    |I_1| \mc \sqrt{t_n^{1/2}t_n} = t_n^{3/4}.
  $$ 
  Let $J_0$ be the union of the pair of symmetric intervals 
  $Z \setminus I_1$, then $\dist(J_0, 0) \mc t_n^{3/4}$ 
  (non-degeneracy implies $I_1$ is roughly centred on $0$). 
  On $J_0$, 
  $$
    |Df| \ggg t_n^{3/4}
  $$
  so,
  on the same set, 
  $$
    |D\phi_1| 
    \ggg t_n^{-1} t_n^{3/4} 
    = t_n^{-1/4}
    .
  $$
  By \rlem{after0},  $t_n^{-1} \mc |Df_0^n(f_0(0))|$, and exponential growth 
  of the latter implies the existence of an $\eta > 0$ for which 
  $t_n < \exp(-5n\eta)$ (for all $n$).
  Combined with the previous sentence, we obtain~\ref{lemphi:b}. 

  It remains to prove~\ref{lemphi:c}. 
  On $I_1$, we claim 
  \begin{equation}\label{eqn:dphi1}
    \frac{D\phi_1(x)}{ a_0 x D\psi_1(f(0))} = 1+o(1)
    \quad \text{as } n \to \infty,
  \end{equation}
  where $a_0 = D^2f_0(0) \ne 0$. 
  Observe that $\phi_1 = \psi_1 \circ f$ and $f(I_1)$ is mapped by $\psi_1$ into $Z$.
  We have $|Z| < t_n^{1/3} \dist(Z, \partial U_1)$;
  by the Koebe Principle, $\psi_1 \colon f(I_1) \to Z$
  has distortion bounded by $1+o(1)$.
  By continuity, $D^2f(x) = a_0(1+o(1))$. 
  Integrating, $Df(x) = a_0x(1+o(1))$, which gives the claim.
        
  This time, integrate $D\phi_1$ to get 
  $$
    \phi_1(x) = \frac{bx^2}{2}(1+(o(1))
  $$
  with $b = a_0 D\psi_1(f(0)).$ The fixed point $y$ in $\partial V$ satisfies 
  $$
    |y| = |\phi_1(y)| = \frac{|b| y^2}{2}(1+(o(1))
    ,
  $$
  so $|y| = \frac{2}{|b|}(1+o(1)).$
  Inserting this in \eqref{eqn:dphi1} gives $|D\phi_1(y)|/2 = 1+o(1)$ and
  $$ D\phi_1(x) = \frac{x}{y} D\phi_1(y)(1+o(1)).$$ 
  If $y'$ is the other boundary point of $V$, then $y' = -y(1+o(1)).$ 
  Hence on $Z \setminus V$, $|x/y| \geq 1+o(1)$ and  
  $\log |D\phi_1(x)| > \log 2 - 1/10 > 1/2$.
\end{proof}

Let $\chi \colon U_1 \setminus V \to U_1 \setminus Z$ be the first entry map to
$U_1 \setminus Z$. By \rlem{phi}, it is well-defined (almost surely).
On $U_1 \setminus Z$ it is identity, while on $Z \setminus V$ it has countably
many branches, each being mapped diffeomorphically onto a connected
component of $U_1 \setminus Z$.

We define $F \colon U_1 \to U_1$ by 
\begin{equation}
    \label{eq:F}
  F(x) = \begin{cases}
    \phi_1 \circ \chi(x), & x\in U_1 \setminus V, \\
      A(x), & x \in V,
  \end{cases}
\end{equation}
where $A$ is an affine homeomorphism between $V$ and $U_1$. 
Let $\tau \colon U_1\setminus V \to \N$ be the corresponding inducing time,
so $F (x) = f^{\tau(x)}(x)$, and set $\tau = 1$ on $V$.

\begin{lem}\label{lemPsin} \
  \begin{enumerate}[label=(\alph*)]
    \item\label{lemPsin:a}  
        All branches of all iterates of $F$ have uniformly bounded distortion (independent of the iterate and of $n$).  
      The image of such a branch is  $U_1$. 
    \item\label{lemPsin:b}
      There exists a constant $\alpha >0$, independent of $n$, so that
      $$
        m(\tau = j) \lll \exp(-\alpha \sqrt{j})
        \qquad \text{for all } j.
      $$
        \end{enumerate}
\end{lem} 

\begin{proof}
  To prove~\ref{lemPsin:a}, it is enough to show  
  that  branches of $F$ other than $V$ are mapped onto $U_1$ and are
  $\Delta$-extensible, with extension contained in $U_1\setminus V$. Let us do this. 
  By \rlem{Ut2}, this holds for branches of $\phi_1$ contained in $U_1\setminus Z$. 
  Each branch of $\chi$ is mapped diffeomorphically by $\chi$ onto 
  a connected component of $U_1 \setminus Z$ and~\ref{lemPsin:a} follows.

  Now we prove~\ref{lemPsin:b}.
  Set $I_0 := Z$ and, inductively,  
  $$
    I_{k+1} := \phi_1^{-1}(I_k) \cap Z.
  $$
  These are nested intervals whose intersection (over all $k$) is $V$. 
  Denote by $J_k$ the pair of symmetric intervals $I_k \setminus I_{k+1}$. 
  On each $J_k$, $\chi = \phi_1^{k+1} =  f^{p(k+1)}$. 
  By \rlem{phi},
  $\log|D\phi_1| \ggg n$ on $J_0$
  and 
  $\log|D\phi_1| > 1/2$ on $J_k$.
  Thus with some $\alpha' > 0$, on $J_k$, 
  $$
    |D\chi| \ggg \exp(\alpha'(n + k))
    .
  $$
  If we take $\alpha'$ small enough, we also have,
  by \rlem{ER}, 
  \[
    m(\{x \in U_1 \setminus Z \colon \tau(x) = j\}) \lll \exp(-\alpha' j)
    .
  \]
  Since $F = \phi_1 \circ \chi = \phi_1 \circ f^{p(k+1)}$ on $J_k$,
 then 
    $$m(\{ x \in J_k \colon \tau(x) = j \}) = 0
    \quad \text{if $k \geq \lfloor j / p \rfloor$ }
    $$
    while, if $k \leq \lfloor\frac{j}p \rfloor -1$, 
  \begin{align*}
    m(\{ x \in J_k \colon \tau(x) = j \})
    & \lll \exp\bigl( - \alpha' (n + k) -\alpha'(j-p(k+1))\bigr)
    %\\ 
    %& =  \exp\bigl(-\alpha'(n-p + j)\bigr)\exp\bigl(k\alpha'(p-1)\bigr)
    .
\end{align*}
Observe that, letting  $k$ go from $1$ to $\lfloor j / p \rfloor -1$, the above forms a geometric sequence in $k$ with ratio $\exp(\alpha'(p-1))$. Its sum is approximated (to within a multiplicative constant) by its maximal term, that is with $k = \lfloor j / p  \rfloor -1$.
    Hence
  \begin{align*}
    m(\{ x \in Z \setminus V \colon \tau(x) = j \})
    & =
    \sum_{k\geq 0} 
    m(\{ x \in J_k \colon \tau(x) = j \}) 
    \\
    & =
    \sum_{k=0}^{\lfloor j / p \rfloor -1}
        m(\{ x \in J_k \colon \tau(x) = j \}) 
        \\
    & \lll \exp(-\alpha'(n + {j}/p))
    \\
    & \lll \exp ( - \alpha'( n + j / n) ) 
     \leq \exp ( - 2 \alpha' \sqrt{j} )
    ,
\end{align*}
using $n+1 \leq p \leq n + N$ to pass to the last line.
  This proves~\ref{lemPsin:b}.
\end{proof}

\begin{lem}
  For every $C > 0$, there is $\delta > 0$
  such that for all sufficiently large $n$,
  \[
    m \bigl( \{ x \in U_1 \colon f^k(x) \in V 
    \text{ for some } k \leq C t_n^{-1} \} \bigr)
    \geq \delta
    .
  \]
\end{lem}

\begin{proof}
  We redefine $f$ on $V = V_n$ so that $f \colon V \to U_1$ is
  the affine homeomorphism $A$ as in~\eqref{eq:F}. This does not change when a point first enters $V$ (noting $k=0$ is possible)
  and does not change $F$.
  With this modification, $F$ is the induced map for
  $f$ with inducing time $\tau$.
  Let $\tau_k = \sum_{j=0}^{k-1} \tau \circ F^j$.

  Let $\nu$ be the Lebesgue measure on $U_1$, normalized so that $\nu(U_1) = 1$.
  Let
  \begin{align*}
    W_k  & = \{x \in U_1 : f^j(x) \not \in V \text{ for all } j \leq k \}, \\
    W'_k & = \{x \in U_1 : f^j(x) \not \in V \text{ for all } j \leq \tau_k \}
    .
  \end{align*}
  By \rlem{Psin}, all branches of $F$ are full and have universally bounded
  distortion. Consequently, the set of points not entering $V$ in $k$ iterates of
  $F$ is exponentially small, namely
  \[
    \nu (W'_k) 
    \leq (1-C_1 |V|)^k
    ,
  \]
  where $C_1$ is a universal constant.
  Now, $W_k \subset W'_\ell \cup \{ \tau_\ell > k \}$ for all $\ell \geq 0$.
  Hence
  \begin{align*}
    \nu(W_k) 
    & \leq \nu(W'_\ell) + \nu(\{\tau_\ell > k\})
    \\ & \leq (1-C_1 |V|)^\ell + \nu(\{\tau_\ell > k\})
    .
  \end{align*}
  We claim that there exists a constant $c>0$ such that
  $\nu(\{\tau_{c k} > k\}) \to 0$ as $k \to \infty$,
  uniformly in $n$. Suppose that the claim is true.
  Setting $k = C t_n^{-1}$ and $\ell = ck$, and using $|V| \mc t_n$,
  we obtain
  \[
    \nu(W_{C t_n^{-1}}) 
    \leq (1- a t_n)^{b t_n^{-1}} + o(1)
    = e^{-ab} + o(1)
    .
  \]
  with some $a,b > 0$. This implies the result.
  
  It remains to verify the claim. The map 
  $F \colon U_1 \to U_1$ is Gibbs-Markov with full images.
  By Lemmas~\ref{lemMisEstF} and~\ref{lemPsin},
  the expansion and distortion bounds of $F$ can be chosen
  independent of $n$. Let $\mu$ be the $F$-invariant
  absolutely continuous probability measure on $U_1$,
  and let $\bar{\tau} = \int \tau \, d\mu$.
  Observe that $\tau$ is constant on the branches of $F$,
  and by \rlem{ER}, $|\tau|_{L^2(\mu)} \mc 1$.
  It is standard (see \rlem{tau}) that
  \[
    |\tau_k - k \bar{\tau}|_{L^2(\mu)}
    \lll k^{-1/2}.
  \]
  It is also standard that $d\mu / d\nu \mc 1$, so
  $|\tau_k - k \bar{\tau}|_{L^2(\nu)} \lll k^{-1/2}$,
  which implies the claim.
\end{proof}

Let $a>0$ and take $C < a/2$. By the preceding lemma, there is a set of measure $\delta >0$ of points which enter $V$ in fewer than $at^{-1}/2$ iterates. Applying \rlem{Wlambda}, $\sbar_{t_n, a t_n^{-1}}
\, 1_{W_\Lambda}(x) \geq (1 - \eps)/2$
    for every $x$ in this set, provided $n$ is large enough. This proves \rprop{DisconReduc} with $\alpha_0 = \delta(1-\eps)/2.$ 

\section{Persistence of statistical stability} \label{sectCon}

In this section we prove \rthm{Con}.
Our strategy is as follows:
\begin{itemize}
  \item {}[\rprop{ContInd} and \S\ref{sec:aNUE}]
    We construct a particular inducing scheme
    for $f_t$, which we use to approximate $f_t$
    with a nonuniformly expanding map $\hf_t$
    which admits an absolutely continuous invariant
    probability measure $\hmu_t$. The construction is such that
    $\hf_0 = f_0$ and $\hmu_0 = \mu_0$.
    The map $\hf_t$ has uniform in $t$ bounds on return times,
    expansion and distortion.
    Further, $\hf_t$ agrees with $f_t$ everywhere except on a set of Lebesgue measure
    of order $t$.
  \item {}[\rlem{AprHat}] Suppose that $\pot \colon I \to \R$ is Lipschitz.
    We show that for all $n \geq 1$,
    \begin{equation}
      \label{eq:all9}
      \int_I
        \Bigl| \frac{1}{n} \sum_{j=0}^{n-1} \pot \circ \hf_t^j
        - \int \pot \, d\hmu_t \Bigr|
      \, dm
      \leq C n^{-1/2} |\pot|_\Lip,
    \end{equation}
    where the constant $C$ does not depend on $t$ and
    $|\cdot|_\Lip$ is the Lipschitz norm,
    \[
      |\pot|_\Lip
      = \sup_{x \in I} |\pot(x)|
      + \sup_{x \neq y \in I} \frac{|\pot(x) - \pot(y)|}{|x-y|}
      .
    \]
  \item {}[\rlem{TrTr}] We show that $f_t$ agrees with $\hf_t$ on time horizons
    smaller than $t^{-1}$, namely that if $n(t) = o(t^{-1})$, then
    \[
      \lim_{t \to 0}
      m \{ x \in I \colon f_t^j(x) = \hf_t^j(x)
        \text{ for all } j \leq n(t) \}
      = 1
      .
    \]
    For a bounded observable $\pot \colon I \to \R$, this naturally implies that
    \begin{equation}
        \label{eq:potffhat}
      \lim_{t \to 0}
      \int_I
        \Bigl|
        \frac{1}{n(t)} \sum_{j=0}^{n(t)-1} \pot \circ \hf_t^j
        -\frac{1}{n(t)} \sum_{j=0}^{n(t)-1} \pot \circ f_t^j 
        \Bigr|
      \, dm \\
      = 0
      .
  \end{equation}
  \item {}[\rlem{StatStab}]
    Using continuity of the map $(x,t) \mapsto f_t(x)$ and \eqref{eq:all9},
    we prove that
    \begin{equation}
        \label{eq:potpotmutmu}
      \int \pot \, d\hmu_t \to \int \pot \, d\mu_0
      \quad \text{as} \quad
      t \to 0
      .
  \end{equation}
\end{itemize}
  From this point, all is straightforward. By~\eqref{eq:all9} and~\eqref{eq:potpotmutmu}, if $n(t) \to \infty$ as $t \to 0$,

      \[
          \lim_{t\to 0} \int_I
        \Bigl| \frac{1}{n(t)} \sum_{j=0}^{n(t)-1} \pot \circ \hf_t^j
        - \int \pot \, d\mu_0 \Bigr|
      \, dm  = 0.
  \]
    %\begin{align*}
      %\int_I
        %\Bigl| \frac{1}{n(t)} & \sum_{j=0}^{n(t)-1} \pot \circ \hf_t^j
        %- \int \pot \, d\mu_0 \Bigr|
      %\, dm \\
      %& \leq 
      %\int_I
        %\Bigl| \frac{1}{n(t)} \sum_{j=0}^{n(t)-1} \pot \circ \hf_t^j
        %- \int \pot \, d\hmu_t \Bigr|
      %\, dm
      %+ \Bigl| \int \pot \, d\mu_0 - \int \pot \, d\hmu_t \Bigr|
      %\\ & 
      %= O(n(t)^{-1/2}) + o(1) = o(1)
      %.
    %\end{align*}
Combining this with~\eqref{eq:potffhat}, we obtain that for all
    Lipschitz $\pot \colon I \to \R$ and $n(t)$
    with $\lim_{t \to 0} n(t) = \infty$ and
    $n(t) = o(t^{-1})$,
    \[
      \lim_{t \to 0}
      \int_I
        \Bigl| \frac{1}{n(t)}
        \sum_{j=0}^{n(t) - 1} \pot \circ f_t^j
        - \int \pot \, d\mu_0 \Bigr|
      \, dm
      = 0
      .
    \]
    This gives the result of \rthm{Con} for Lipschitz observables.
  Generalisation to the class of continuous observables is
    automatic: every continuous observable
    can be arbitrarily well approximated by a Lipschitz
    observable in the uniform topology.     

In the rest of this section we implement the strategy above.
Where there is no ambiguity, we suppress the dependence on $t$.

\begin{rem}
  \label{remSquareRootT}
  One of the main difficulties in our proof is the construction of
  the approximating map $\hf_t$, which allows a suitable inducing scheme
  and which coincides with $f_t$ everywhere except on a set of Lebesgue measure of order $t$.

  The proof for time horizons of order $o(t^{-1/2})$ can be made
  significantly simpler than that for $o(t)$, as we only have to avoid a set of size $t^{-1/2}$.
  In this case, $\hf_t$ can be taken to be equal to $f_t$ everywhere outside
  the central branch $Z$ of the first return map to $U_0$, if such a branch exists; on
  $Z$, we can define $\hf_t$ as an affine bijection between $Z$ and $U_1$.

  Then it can be verified, using results of \rsec{Prel} and \rsec{SetUp},
  that the first return map to $U_1$
  under $\hf_t$ has all of its branches full with universally bounded distortion,
  and a uniform in $t$ exponential bound on return times. Similarly to \rlem{phi},
  one can show that $|Z| \lll t^{1/2}$.
  With this, the strategy above works, rendering unnecessary most of
  \rsec{IndScheme}.
\end{rem}

\subsection{Inducing scheme}
\label{sectIndScheme}

Recall that $\phi_1 \colon U_1 \to U_1$ is the first return map
under $f$. It is constructed to have countably many branches, and
all non-central branches (i.e.\ not containing $0$)
are mapped by $\phi_1$ to $U_1$ diffeomorphically.

Let $V = (-Ct, Ct)$, where $C$ is the constant from \rlem{BigDer}.
Then $|D\phi_1| > 1000$ on $U_1 \setminus V$.

\begin{prop}
  \label{propContInd}
  For small enough $t$, there exists a partition $\cP$ of $U_1$
  into open intervals, modulo a zero measure set.
  Each interval $J \in \cP$ is coloured blue or red, and there is a function 
  $\rho \colon U_1 \to \N \cup \{0\}$, constant on each $J$ with value $\rho(J)$,
  such that:
  \begin{enumerate}[label=(\alph*)]
    \item\label{propContInd:x} if $J$ is red, then $f^{\rho(J)} (J) \subset V$ ;
    \item\label{propContInd:xx} if $J$ is blue, then $\rho(J) > 0$ and
      $f^{\rho(J)} \colon J \to U_1$ is a
      diffeomorphism with universally bounded distortion;
    \item\label{propContInd:y} $m(\cup\{J \in \cP \colon J \text{ is red}\}) \lll t$;
    \item\label{propContInd:z} $\int_{U_1} \rho^2 \, dm \mc 1$.
  \end{enumerate}
\end{prop}

The proof of \rprop{ContInd} takes the rest of this subsection.
To simplify notation, if $W$ is a branch of $\phi_1$ intersecting $\partial V$, we consider the connected components of $W\setminus \partial V$ 
 as separate branches of $\phi_1$.
In particular, if $W'$ is a branch of $\phi_1^k$, then $\phi_1^p(W') \cap \partial V = \emptyset$ for $0\leq p<k$.

Let $\tau \colon U_1 \to \N$ be the first return time,
\[
  \tau(x) = \inf \{k \geq 1 \colon f^k(x) \in U_1 \}
  ,
\]
so $\phi_1 = f^\tau$.
Let $\tau_k = \sum_{j=0}^{k-1} \tau \circ \phi_1^j$.
Note that if $W$ is a branch of $\phi_1^k$, then
as a consequence of \rlem{Ut0}\ref{lemUt0:bdd},
for each $j \leq \tau_k(W)$ either
$f^j(W) \subset U_0$ or $f^j(W) \cap U_0 = \emptyset$.

We construct a nested sequence of partitions
$\cP_k$, $k \geq 0$, of $U_1$ into open intervals.
To each interval we assign a colour (yellow, blue or red),
an index and a height (integers).  
Let $\cP_0 = \{U_1\}$ be the trivial partition. 
      We set the height
    of its only element to $0$, index to $0$ and colour it yellow. 
   For $k \geq 1$, we construct $\cP_k$ as a refinement of $\cP_{k-1}$ inductively:
\begin{itemize}
  \item 
    We leave the blue and red intervals intact,
    with the same height and index.
  \item
    We partition each yellow $J \in \cP_{k-1}$ into the branches
    of the map $\phi_1^k \colon J \to U_1$.
    For each such new element $W$ of $\cP_k$:
    \begin{itemize}
      \item If $\phi_1^{k-1}(W) \subset V$, then we colour $W$ red. Otherwise,
        $\phi_1^{k-1}(W) \cap V = \emptyset$.
        If $\phi_1^k \colon W \to U_1$ is a $U_0$-extensible
        diffeomorphism, we colour $W$ blue.
        Otherwise we colour $W$ yellow.
      \item We set 
        \[
          \hei(W) = \begin{cases} 
            k-1, & W \text{ is red} \\
            k, & \text{otherwise}
          \end{cases}
        \]
        and
        \[
          \ind(W) = \# \{0 < j \leq \tau_{\hei(W)}(W) \colon f^j(W) \subset U_0 \}
          .
        \]
    \end{itemize}
\end{itemize}

\begin{lem}
  \label{lemNahh}
  For all \(\ell \geq 0\),
  \begin{itemize}
    \item
      \(
        \sum_{k \geq 0} \# \bigl\{ 
            J \in \cP_k \colon J \text{ is yellow with index } \ell
        \bigr\}
        \leq 6^\ell
        .
      \)
    \item
      \(
        \sup_{k \geq 0} \# \bigl\{ 
            J \in \cP_k \colon J \text{ is red with index } \ell
        \bigr\}
        \leq 6^\ell
        .
      \)
  \end{itemize}
\end{lem}

\begin{proof}
  Suppose that $J \in \cP_{k-1}$ is yellow with index $\ell$.
  In $\cP_k$ it is partitioned into subintervals.
  We claim that among these:
  \begin{enumerate}[label=(\alph*)]
    \item\label{lemNahh:red} there is at most \(1\) red interval, its index is $\ell$;
    \item\label{lemNahh:yellow} all yellow intervals have index at least $\ell+1$, and 
      there are at most \(4\) of them with index \(\ell + j\) for each \(j \geq 1\).
  \end{enumerate}
  A recursive estimate then implies that the number of yellow intervals contributing to  the above sum is bounded by $6^\ell$. The same estimate holds then for red intervals and the  result follows. We justify the claim now. 
  
  To each branch of $\phi_1^k$ contained in $J$ corresponds a branch of
  the restriction $\phi_1 \colon \phi_1^{k-1}(J) \to U_1$. The red interval corresponds to $V$ intersected with $\phi_1^{k-1}(J)$. The statement of \ref{lemNahh:red} is immediate.
  
  Let $\hat J$ be a connected component of $\phi_1^{k-1}(J)\setminus V$.
  Let $W$ be a branch of the restriction $\phi_1 \colon \hat J \to U_1$
  with $\tau = n$ on $W$. To $W$ corresponds the element $\hat W := \phi_1^{-(k-1)}(W) \cap J$ of $\cP_k$,
  which is yellow or blue.
  
   We call $W$ \emph{unobstructed} if 
  $f^n\colon W \to U_1$ is a diffeomorphism and there is
  an open interval $W_0 \subset \hat J$, compactly containing $W$,
  such that $f^n\colon W_0 \to U_0$ is a diffeomorphism.
  Otherwise $W$ is obstructed. Note that obstruction depends on $\hat J$ and that $\hat W$ can only be yellow if $W$ is obstructed. 
            
  Let us examine the case when $W$ is obstructed.
  There are $w \in \partial W$ and $v \in \ol{\hat J} \setminus W$ with $[w,v] \cap W = \emptyset$ 
  such that (noting $v$ and $w$ may coincide)
  \begin{itemize}
      \item $f^n$ is monotone on $W \cup [w,v]$,
      \item $f^n([w,v])$ does not contain a connected component of $U_0\setminus {U_1},$
    \item either $Df^n(v) = 0$ or $v \in \partial \hat J$.
  \end{itemize}      
  Since $f^n([w,v])$ does not contain a connected component of
  $U_0\setminus {U_1}$, it follows 
          (via \rlem{Ut0}\ref{lemUt0:bdd})
  that $f^p([w,v])$ does not
  contain a point of $\partial U_0 \cup \partial U_1$
  for all $0\leq p < n$. This implies that
  $f^p([w,v]) \cap U_1 = \emptyset$ for all $0 < p <n$.
  Therefore $Df^n(v) \ne 0$, so $v \in \partial \hat J$. 

  As $f^n$ is monotone on $W \cup [w,v]$, 
      there is a one-to-one correspondence between obstructed branches $W \subset \hat{J}$ of $\phi_1$ and a subset of the set of pairs $(v,n) \in \partial \hat{J} \times \N$ 
      for which $f^n(v) \in U_0$. 
      For each such $W$ and associated $(v,n)$, there is a unique 
      $j(v,n) := \#\{0\leq p \leq n : f^p(v) \in     U_0\}$. 
  Moreover, for $0 \leq p \leq n$, either $f^p(W \cup [w,v]) \subset U_0$ or
  $f^p(W \cup [w,v]) \cap U_0 = \emptyset$, from which it follows that 
  $$j(v,n) = \#\{0\leq p \leq n : f^p(W) \subset U_0\}.$$ 
  Hence to each yellow element $\hat W \subset J$ in $\cP_k$, there is a unique obstructed branch $W$ with associated $\hat J$ and pair $(v,n)$. The index of $\hat W$ is $\ell + j(v,n)$.
      With at most two ways to choose $\hat{J}$ as a connected component of $\phi_1^{k-1}(J)\setminus V$,  and two possibilities for $v \in \partial \hat J$,  the claim and \ref{lemNahh:yellow} follow.
\end{proof}

\begin{lem}
  \label{lemAnfg}
  \[
    \sup_{n \geq 0} \sum_{\substack{J \in \cP_n , \\ J \text{ is red}}} |J|
    \lll t
    \qquad \text{and} \qquad
    \sum_{n \geq 0} \sum_{\substack{J \in \cP_n , \\ J \text{ is yellow}}} |J|
    \lll 1
    .
  \]
\end{lem}

\begin{proof}
  Suppose that $J \in \cP_n$ is an interval with index $\ell$ and height $h$.
  By \rlem{BigDer}, the first return map to $U_0$, restricted to $U_0 \setminus V$,
  is expanding by a factor of at least $1000$. By construction, $\phi_1^k(J)$
  does not intersect $V$ for $k < h$. Thus:
  \begin{itemize}
    \item $|D \phi_1^h| \geq 1000^\ell$ on $J$, and $\phi_1^h(J) \subset U_1$,
      so 
      \[
        |J| \lll 1000^{-\ell}
        ;\]
    \item moreover, if $J$ is red, then $\phi_1^h(J) \subset V$,
      so 
      \[
        |J| \lll 1000^{-\ell} |V| \lll 1000^{-\ell} t
        .
      \]
  \end{itemize}
  By \rlem{Nahh}, $\cP_n$ has at most $6^\ell$ red intervals of index $\ell$,
  thus
  \[
    \sum_{\substack{J \in \cP_n , \\ J \text{ is red}}} |J|
    \lll \sum_{\ell \geq 0} 6^\ell 1000^{-\ell} t
    \lll t
    .
  \]
  The result for red intervals follows. The argument for yellow intervals is similar.
\end{proof}

Let $\cP = \vee_n \cP_n$. By \rlem{Anfg}, $\cP$ is a partition of
$U_1$ into open intervals (blue and red), modulo a zero measure set.
For $J \in \cP$, let $\rho(J) = \tau_{\hei(J)}$. This defines $\rho \colon U_1 \to \N$
with value $\rho(J)$ on each $J \in \cP$. 

By construction, $\rho$ satisfies 
\ref{propContInd:x}, \ref{propContInd:xx} and \ref{propContInd:y} of \rprop{ContInd}.
It remains to prove \ref{propContInd:z}.

\begin{lem}
  $\int_{U_1} \rho^2 \, dm \mc 1$.
\end{lem}

\begin{proof}
  It is clear that $\int_{U_1} \rho^2 \, dm \ggg 1$.
  
  Let $J \in \cP$, so $J$ is red or blue. Let $h = \hei(J)$ and for $k \leq h$, let $J_k$ be
  the element of $\cP_k$ containing $J$. Each $J_k, k <h$, is yellow, while $J_h$ is yellow or blue.
  Then
  \[
    \rho(J) = \sum_{k=0}^{h-1} \tau \circ \phi_1^k (J_{k+1})
    .
  \]
  %The sequence $\{\ind(J_k)\}$ is strictly increasing with $k$ (but not necessarily by $1$).
  Define  $\rho_i$ at a point $x$ by: $\rho_i(x) = \tau \circ \phi_1^{k}(x)$ if $x$ is contained in a yellow interval $J' \in \cP_k$ with index $i$ and height $k$, for some $k$, but $x$ is not contained in a red interval of height $k$ (in $\cP_{k+1}$), and $\rho_i(x) = 0$ otherwise. 
  Then 
  \[
      \rho = \sum_{i=0}^\infty \rho_i.
  \]
  %\[
    %\rho_i(J) = \sum_{k=0}^{h-1} \begin{cases}
      %\tau \circ \phi_1^k(J_{k+1}) , & \text{if } \ind(J_k) = i, \\
      %0, & \text{else.}
    %\end{cases}
  %\]
  %Note that $\rho_0 = \tau$ on $U_1 \setminus V$, for example.

  Let $J \in \bigcup_{n\geq 0}\cP_n$ be yellow.
  The map $\phi_1^{\hei(J)} \colon J \to U_1$
  is monotone and, following the proof of \rlem{Anfg}, 
  it is expanding by a factor of at least $1000^{\ind(J)}$.
  Using \rlem{ER},
  \[
    \int_{J} \tau^2 \circ \phi_1^{\hei(J)} \, dm
    \lll 1000^{-\ind(J)} \int_{U_1} \tau^2 \, dm
    \lll 1000^{-\ind(J)}
    . 
  \]
  Let $i \geq 0$. Let $\cA_i := \{J \in \bigcup_{n\geq0} \cP_n : J \text{ is yellow with index $i$}\}$. 
  By \rlem{Nahh}, $\#\cA_i \leq 7^i$.
  Observe that
  \[
    \rho_i
    = \sum_{J \in \cA_i} \tau \circ \phi_1^{\hei(J)} \big|_J
    .
  \]
  The elements of $\cA_i$ are pairwise disjoint, thus
  \[
    \int_{U_1} \rho_i^2 \, dm
    = \sum_{J \in \cA_i} \int_J \tau^2 \circ \phi_1^{\hei(J)} \, dm
    \lll 7^i \cdot 1000^{-i} \leq 100^{-i}
    . 
  \]
  Finally,
  \[
    \Bigl[ \int_{U_1} \rho^2 \, dm \Bigr]^{1/2}
    \lll \sum_{i=0}^\infty
    \Bigl[ \int_{U_1} \rho_i^2 \, dm \Bigr]^{1/2}
    \lll 1
    .
  \]
\end{proof}

\subsection{Approximation with nonuniformly expanding map}
\label{sec:aNUE}

Let $\cP$ be the partition given by \rprop{ContInd}.
For an interval $J \subset U_1$, let $\hf_J \colon J \to U_1$ be
a linear bijection.
Define $\hf \colon I \to \R$ and $\hrho \colon U_1 \to \N$,
\begin{align*}
  \hf(x) & = \begin{cases}
    \hf_J(x), & \text{if } x \in J, \ J \in \cP \text{ is red}, \\
    f(x), & \text{else},
  \end{cases}
  \\
  \hrho(x) & = \begin{cases}
    1, & \text{if } x \in J, \ J \in \cP \text{ is red}, \\
    \rho(x), & \text{else}.
  \end{cases}
\end{align*}

Let $\hF \colon U_1 \to U_1$, $\hF(x) = \hf^{\hrho(x)}(x)$.
In particular, $\hF$ coincides with $f^\rho$ on all blue elements of $\cP$.
Our construction ensures that
there are constants $C > 0$ and $\lambda > 1$, independent of $t$,
such that for every $J \in \cP$ and $x,y \in J$:
\begin{itemize}
  \item the restriction $\hF \colon J \to U_1$ is a bijection;
  \item $|\hF(x) - \hF(y)| \geq \lambda |x-y|$;
  \item
    $
      \bigl| \log |D\hF(x)| - \log |D\hF(y)| \bigr|
      \leq C |\hF(x)-\hF(y)|
    $;
  \item $|\hf^j(x) - \hf^j(y)| \leq C |\hF(x)-\hF(y)|$
    for all $0 \leq j \leq \hrho(J)$;
  \item $\int_{U_1} \hrho^2 \, dm \leq C$.
\end{itemize}
That is, $\hf$ is a \emph{nonuniformly expanding} map as in
Appendix~\ref{sectNUE:M}.
There is a unique absolutely continuous $\hf$-invariant
probability measure $\hmu$.

\begin{lem}
  \label{lemAprHat}
  For all Lipschitz $\pot \colon I \to \R$ and $n \geq 1$, 
  \[
    \int_I
      \Bigl| \frac{1}{n} \sum_{j=0}^{n-1} \pot \circ \hf^j
      - \int \pot \, d\hmu \Bigr|
    \, dm
    \leq C n^{-1/2} |\pot|_\Lip,
  \]
  where the constant $C$ does not depend on $t$.
\end{lem}

\begin{proof}
  By \rlem{Mommu},
  \begin{equation}
    \label{eq:veryrandom}
    \int
      \Bigl| \frac{1}{n} \sum_{j=0}^{n-1} \pot \circ \hf^j
      - \int \pot \, d\hmu \Bigr|
    \, d\hmu
    \lll n^{-1/2} |\pot|_\Lip
    .
  \end{equation}
  Note that the integral is taken with respect to the
  invariant measure $\hmu$ rather than $m$. It remains to establish
  an appropriate connection between $m$ and $\hmu$.
  For this, we follow \cite{K18eq}.
  
  Let $\psi_1 \colon I \to U_1$ be the first entry map for $f$
  (the same as for $\hf$) and $\tau \colon I \to \N \cup \{0\}$,
  \[
    \tau(x) = \inf \{k \geq 0 \colon f^k(x) \in U_1\}
    ,
  \]
  so that $\psi_1(x) = \hf^{\tau(x)}(x) = f^{\tau(x)}(x)$.
  
  It follows from \rlem{ER} that $\int_I \tau \, dm \lll 1$.
  Since $f(\partial I) \subset \partial I$ and
  $f^j(\partial U_1) \cap U_1 = \emptyset$ for all $j$,
  every branch of $\psi_1$ is mapped
  diffeomorphically on $U_1$. By \rlem{DistExp},
  $\psi_1$ has universally bounded distortion.
  
  Write
  $
    m = \sum_{J \in B} m(J) m_J
  $,
  where $B$ is the set of all branches of $\psi_1$ and $m_J$ is the
  normalized to probability restriction of $m$ to $J$.
  For each $J$, the probability measure $f^{\tau(J)}_* m_J$
  is supported on $U_1$, and due to the bounded distortion,
  it is regular in the sense of \cite{K18eq},
  with the regularity constant ($R'$ in \cite{K18eq})
  independent of $t$.
  Thus $m$ is \emph{forward regular}. The
  \emph{jump function} $\tau \colon B \to \N \cup \{0\}$
  has bounded (uniformly in $t$) first moment:
  $\sum_{J \in B} m(J) \tau(J) \lll 1$.
  
  Let $X_n$ and $Y_n$ the the discrete time random processes
  given by $\sum_{j=0}^{n-1} \pot \circ \hf^j$ on the
  probability spaces $(I,m)$ and $(I, \hmu)$ respectively.
  By \cite[Thm. 2.5]{K18eq}, there is a coupling of
  $X_n$ and $Y_n$,
  that is, there exists a probability space $\Omega$ supporting
  random processes $\{X'_n\}$ and $\{Y'_n\}$, equal in distribution
  to $\{X_n\}$ and $\{Y_n\}$ respectively, such that
  \begin{equation}
    \label{eq:coup}
    \E \left(\sup_{n \geq 0} |X_n' - Y_n'|\right)
    \lll \sup_I |\pot|
    .
  \end{equation}
  Bound~\eqref{eq:coup},
  together with~\eqref{eq:veryrandom}, implies our result.
\end{proof}

Let
$
  I_r = \cup \{ J \in \cP \colon J \text{ is red} \}
$.

\begin{lem}
  \label{lemTrTr}
  There is a constant $C > 0$, independent of $t$,
  such that for all $n \geq 0$,
  \[
    m \{ x \in I \colon f^j(x) \not \in I_r \text{ for all } j \leq n \}
    \ggg (1-Ct)^n
    .
  \]
  In particular, if $n(t) = o(t^{-1})$, then 
  \[
    \lim_{t \to 0}
    m \{ x \in I \colon f^j(x) = \hf^j(x) \text{ for all } j \leq n(t) \}
    = 1
    .
  \]
\end{lem}

\begin{proof}
  Let $\tau \colon I \to \N$,
  \[
    \tau(x) = \inf \{k \geq 1 \colon f^k(x) \in U_1\}
  \]
  and $g \colon I \to U_1$, $g(x) = f^{\tau(x)}(x)$.

  Observe that
  \[
    m \{ x \in I \colon f^j(x) \not \in I_r \text{ for all } j \leq n \}
    \geq
    m \{ x \in I \colon g^j(x) \not \in I_r \text{ for all } j \leq n \}
    .
  \]
  By \rprop{ContInd}, all branches of the map $g$
  in $U_1 \setminus I_r$ are mapped diffeomorphically
  and with uniformly bounded distortion onto $U_1$.
  So are the branches in $I \setminus U_1$,
  following the argument for the first entry map $\psi_1$ in
  the proof of \rlem{AprHat}.
  \rprop{ContInd} guarantees that $m(I_r) \lll t$.
  Therefore,
  \[
    m \{ x \in I \colon g^n(x) \in I_r \mid
         g^j(x) \not \in I_r \text{ for all } j < n
    \} 
    \lll \frac{m(I_r)}{m(U_1)}
    \lll t
    .
  \]
  The result follows.
\end{proof}

\begin{lem}
  \label{lemStatStab}
  For all Lipschitz $\pot \colon I \to \R$, we have
  $\int \pot \, d\hmu_t \to \int \pot \, d\mu_0$ as $t \to 0$.
\end{lem}

\begin{proof}
  For every (fixed) $n \geq 1$, the map $(x, t) \mapsto f_t^n(x)$ is continuous.
  Thus
  \[
    \sup_I
    \Bigl|
      \frac{1}{n} \sum_{j=0}^{n-1} \pot \circ f_t^j
     -\frac{1}{n} \sum_{j=0}^{n-1} \pot \circ f_0^j
    \Bigr|
    \to 0
    \quad \text{as } t \to 0
    .
  \]
  By \rlem{TrTr}, as $t \to 0$,
  \begin{align*}
    \int_I 
      &
      \Bigl|
        \frac{1}{n} \sum_{j=0}^{n-1} \pot \circ \hf_t^j
      - \frac{1}{n} \sum_{j=0}^{n-1} \pot \circ   f_0^j
      \Bigr|
    \, dm
    \leq 
    \sup_I
      \Bigl|
        \frac{1}{n} \sum_{j=0}^{n-1} \pot \circ f_t^j
       -\frac{1}{n} \sum_{j=0}^{n-1} \pot \circ f_0^j
      \Bigr|
    \\ & + 2 \sup_I |\pot| \,
      m \{ x \in I \colon f_t^j(x) = \hf_t^j(x)
        \text{ for all } j \leq n \}
    = o(1)
    .
  \end{align*}
  By \rlem{AprHat},
  \begin{align*}
    \Bigl| \int \pot \, d\hmu_0 - & \int \pot \, d\hmu_t \Bigr|
    \\
    & \lll \int_I
      \Bigl|
        \frac{1}{n} \sum_{j=0}^{n-1} \pot \circ \hf_t^j
      - \frac{1}{n} \sum_{j=0}^{n-1} \pot \circ   f_0^j
      \Bigr|
    \, dm
    + n^{-1/2} |\pot|_\Lip
    \\ & = o(1) + n^{-1/2} |\pot|_\Lip
    .
  \end{align*}
  Since we can fix $n$ arbitrarily large, the result follows.
\end{proof}

\appendix

\section{Moment estimates for nonuniformly expanding maps}
\label{sectNUE:M}

Let $(M,d)$ be a bounded metric space with a map $f \colon M \to M$.
Suppose that $Y \subset M$ and $m$ is a Borel probability measure on $Y$.
Suppose that $\alpha$ is a finite or countable partition of $Y$ 
(up to a zero measure set) with $m(a) > 0$ for all $a \in \alpha$.
We require that there exist an integrable function
$\tau \colon Y \to \{1,2,\ldots\}$, constant on each $a \in \alpha$ with
value $\tau(a)$, and constants $\lambda > 1$, $K > 0$ and $\eta \in (0,1]$ such that
for each $a \in \alpha$,
\begin{itemize}
  \item $F = f^{\tau}$ restricts to a (measure-theoretic) bijection from
    $a$ to $Y$;
  \item $d(F(x), F(y)) \geq \lambda d(x,y)$ for all $x,y \in a$;
  \item $d(f^\ell(x), f^\ell(y)) \leq K d(F(x), F(y))$
    for all $x,y \in a$ and $0 \leq \ell \leq \tau(a)$;
  \item the inverse Jacobian $\zeta=\frac{dm}{dm \circ F}$ of the restriction
    $F \colon a \to Y$ satisfies
    \[
      \bigl| \log \zeta(x) - \log \zeta(y) \bigr|
      \leq K d(F(x),F(y))^\eta
    \]
    for all $x,y \in a$.
\end{itemize}
We say that $f \colon M \to M$ as above is a \emph{nonuniformly expanding map.} We refer to $Y$ as the \emph{inducing set}, to $\tau$ as
the \emph{inducing time} and to $F$ as the \emph{induced map.}

We assume that $\int_Y \tau^2 \, dm < \infty$. We use $C$ to denote various positive
constants which depend continuously (only) on $\eta$, $K$, $\lambda$, $\diam M$ and
$\int_Y \tau^2 \, dm$.

\begin{lem}[{\cite[Prop.~2.5]{KKM16exp}}]
  There exists a unique $F$-invariant
  probability measure $\mu_Y$ on $Y$, 
  absolutely continuous with respect to
  $m$, and 
  \[
    C^{-1}
    \leq \frac{d\mu_Y}{dm}
    \leq C.
  \]
\end{lem}

Define a Young tower
\[
  \Delta = \{(y, \ell) \in Y \times \Z \colon 0 \leq \ell < \tau(y) \}
\]
with a tower map $T \colon \Delta \to \Delta$,
\[
  T(y, \ell) = \begin{cases}
    (y, \ell + 1), & \ell < \tau(y) - 1, \\
    (F(y),0), & \ell = \tau(y) - 1,
  \end{cases}
\]
and a projection $\pi \colon \Delta \to M$,
$\pi(y,\ell) = f^\ell(y)$. Then $\pi$ is a semi-conjugacy
between $T \colon \Delta \to \Delta$ and $f \colon M \to M$,
i.e.\ $\pi \circ T = f \circ \pi$.

The measure 
\[
    \mu_\Delta = \frac{\mu_Y \times \text{counting} }{ \int \tau \, d\mu_Y}
\]
is a $T$-invariant probability measure on $\Delta$, and
$\mu = \pi_* \mu_\Delta$ is an $f$-invariant probability
measure on $M$.

Suppose that $\pot \colon M \to \R$. Define
\begin{equation} \label{eq:hoho}
  |\pot|_\eta = \sup_{x \neq y \in M} \frac{|\pot(y) - \pot(x)|}{d(x,y)^\eta}
  , \qquad
  |\pot|_\infty = \sup_{x \in M} |\pot(x)|
  , \qquad
  \|\pot\|_\eta = |\pot|_\eta + |\pot|_\infty
  .
\end{equation}
We define similarly $| \cdot |_\eta$, $| \cdot |_\infty$ and $\| \cdot \|_\eta$
for functions $\pot \colon Y \to \R$.

\begin{lem}\label{lemtau}
  Let $\bar{\tau} = \int_Y \tau \, d\mu_Y$ and $\tau_k  = \sum_{j=0}^{k-1} \tau \circ F$.
  Then
  \[
    \bigl| \tau_k - k \bar{\tau} \bigr|_{L^2(\mu_Y)}
    \leq C k^{-1/2}
    .
  \]
\end{lem}

\begin{proof}
  Let $P \colon L^1(\mu_Y) \to L^1(\mu_Y)$ denote the transfer operator
  corresponding to $F$ and $\mu_Y$, so 
  $\int_Y v \circ F \, w \, d\mu_Y = \int_Y v \, P w \, d\mu_Y$
  for all $v \in L^\infty$ and $w \in L^1$.
  
  Let $\pot = \tau - \bar{\tau}$. It is a direct verification that $\|P \pot \|_\eta \leq C$.
  Thus, by \cite[Cor.~2.4]{KKM16exp}, $\|P^k \pot\|_\eta \leq C \gamma^k$ for all $k \geq 1$,
  where $\gamma \in (0,1)$ depends only on $\lambda$, $K$, $\eta$ and $\diam M$.
  
  Finally,
  \[
    \int_Y \Bigl( \sum_{j=0}^{k-1} \pot \circ F \Bigr)^2 \, d\mu_Y
    \leq k \int_Y \pot^2 \, d\mu_Y 
    + 2k \sum_{j=1}^{\infty} \Bigl| \int_Y \pot \circ F^k \, \pot \, d\mu_Y \Bigr|
    \leq C k
    .
  \]
  The result follows.
\end{proof}

\begin{lem}[{\cite[Cor.~2.10]{KKM17mart}}] \label{lemMommu}
  For all $\pot \colon M \to \R$ and $n \geq 0$,
  \[
    \biggl|
    \sup_{k \leq n} 
    \Bigl|
      \sum_{j=0}^{k-1} \pot \circ f^j 
      - k \int \pot \, d\mu
    \Bigr|
    \biggr|_{L^2(\mu)}
    \leq C \|\pot\|_\eta n^{1/2}
    .
  \]
\end{lem}

Observe that $\frac{dm}{d\mu} \leq C$. Thus
\begin{cor}
  For all $\pot \colon M \to \R$ and $n \geq 0$,
  \[
    \biggl|
    \sup_{k \leq n} 
    \Bigl|
      \sum_{j=0}^{k-1} \pot \circ f^j 
      - k \int \pot \, d\mu
    \Bigr|
    \biggr|_{L^2(m)}
    \leq C \|\pot\|_\eta n^{1/2}
    .
  \]
\end{cor}

We define a metric $d_{\Delta}$ on $\Delta$ by
\[
  d_{\Delta}((y, \ell),(y', \ell')) =
  \begin{cases}
    d(y,y'), & \ell=\ell'; \\
    \diam M, & \text{otherwise}.
  \end{cases}
\]
Define $|\cdot|_\eta$, $|\cdot|_\infty$ and $\|\cdot\|_\eta$ for functions on
$\Delta$ similarly to~\eqref{eq:hoho}.

\begin{rem}
  \label{remYiaf}
  $T \colon \Delta \to \Delta$ is itself a nonuniformly expanding map.
  Thus for all $\psi \colon \Delta \to \R$,
  \[
    \biggl|
    \sup_{k \leq n} 
    \Bigl|
      \sum_{j=0}^{k-1} \psi \circ T^j 
      - k \int \psi \, d\mu_\Delta
    \Bigr|
    \biggr|_{L^2(\mu_\Delta)}
    \leq C \|\psi\|_\eta n^{1/2}
    .
  \]
\end{rem}

\end{document}